\documentclass[12pt]{article}
\usepackage{amssymb}
\usepackage{amsmath}
\usepackage{latexsym}
\usepackage{bm}
\usepackage{enumitem}
\usepackage{graphicx}
\usepackage{extarrows}
\usepackage{amsfonts}
\usepackage{amssymb}
\usepackage{indentfirst}
\usepackage{bbm}
\usepackage{xypic}
\usepackage{mathdots}
\usepackage{float}
\usepackage{tikz-cd}
\usepackage{pst-node}
\usepackage{amscd}

\usepackage{tocbibind}
\usepackage{geometry}
\allowdisplaybreaks[4]
\newtheorem{definition}{Definition}[section]
\newtheorem{theorem}{Theorem}[section]
\newtheorem{proposition}[theorem]{Proposition}
\newtheorem{lemma}[theorem]{Lemma}
\newtheorem{corollary}[theorem]{Corollary}
\newtheorem{example}[theorem]{Example}
\newtheorem{remark}[theorem]{Remark}

\DeclareMathOperator{\End}{End}

\DeclareMathOperator{\spn}{span}

\DeclareMathOperator{\Aut}{Aut}

\DeclareMathOperator{\tr}{tr}
\DeclareMathOperator{\Tr}{Tr}
\DeclareMathOperator{\rk}{rank}

\DeclareMathOperator{\Mat}{Mat}
\DeclareMathOperator{\ad}{ad}
\DeclareMathOperator{\Ad}{Ad}
\DeclareMathOperator{\Jac}{Jac}

\begin{document}
\title{Actions of Cusp Forms on Holomorphic Discrete Series and Von Neumann Algebras}

\author{Jun Yang}
\date{}
\maketitle

\abstract{
A holomorphic discrete series  representation $(L_{\pi},H_{\pi})$ of a connected semi-simple real Lie group $G$ is associated with an irreducible representation $(\pi,V_{\pi})$ of its maximal compact subgroup $K$. 
The underlying space $H_{\pi}$ can be realized as certain holomorphic $V_{\pi}$-valued functions on the bounded symmetric domain $\mathcal{D}\cong G/K$.  
By the Berezin quantization, we transfer $B(H_{\pi})$ into $\End(V_{\pi})$-valued functions on $\mathcal{D}$. 
For a lattice $\Gamma$ of $G$, we give the formula of a faithful normal tracial state on the commutant $L_{\pi}(\Gamma)'$ of the group von Neumann algebra  $L_{\pi}(\Gamma)''$. 
We find the Toeplitz operators $T_f$'s associated with essentially bounded $\End(V_{\pi})$-valued functions $f$'s on $\Gamma\backslash\mathcal{D}$ generate the entire commutant  $L_{\pi}(\Gamma)'$: 
\begin{center}
$\overline{\{T_f|f\in L^{\infty}(\Gamma\backslash\mathcal{D},\End(V_{\pi}))\}}^{\text{w.o.}}=L_{\pi}(\Gamma)'$.
\end{center}
For any cuspidal automorphic form $f$ defined on $G$ (or $\mathcal{D}$) for $\Gamma$, we find the associated Toeplitz-type operator $T_f$ intertwines the actions of $\Gamma$ on these square-integrable representations. 
Hence the composite operator of the form $T_g^{*}T_f$ belongs to $L_\pi(\Gamma)'$. 
We prove these operators span $L^{\infty}(\Gamma\backslash\mathcal{D})$ and
\begin{center}
$\overline{\langle\{\text{span}_{f,g} T_g^{*}T_f\}\otimes \End(V_{\pi})\rangle}^{\text{w.o.}}=L_\pi(\Gamma)'$,
\end{center}
where $f,g$ run through holomorphic cusp forms for $\Gamma$ of same types.
If $\Gamma$ is an infinite conjugacy classes group, we obtain a $\text{II}_1$ factor  
from cusp forms. }
\tableofcontents
\section{Introduction}

A holomorphic discrete series representation is an infinite dimensional, square-integrable, irreducible unitary representation of a semi-simple Lie group $G$, which is usually non-compact. 
There is a large family of discrete series representations of real reductive Lie groups and also an interesting family of admissible representations of algebraic groups over $\mathbb{R}$. 
Harish-Chandra proved such representations exist if and only if $\rk{G}=\rk{K}$, where $K$ is a maximal compact subgroup with a non-finite center \cite{HC,Kn}. 
Indeed, such a representation can be realized as certain holomorphic functions on the bounded symmetric domain $\mathcal{D}=G/K$ with values in a highest weight representation $(\pi,V_{\pi})$ of $K$.  
In particular, they reduce to these highest weight representations when $G=K$, i.e., $G$ is compact. 
As is the case with these finite-dimensional highest weight representations, the holomorphic discrete series can also be described by the dominant weights of $K$. 

The first non-trivial example is the discrete series of $SL(2,\mathbb{R})$, whose maximal compact subgroup is $SO(2)$ (up to conjugation). 
In this case, the symmetric domain $\mathcal{D}=SL(2,\mathbb{R})/SO(2)$ is just the open unit disk which is holomorphically isomorphic to the Poincar\'{e} upper half-plane $\mathbb{H}$. 
As the irreducible representations of $SO(2)$ are characterized by integers, the holomorphic discrete series of $SL(2,\mathbb{R})$ can be denoted by $(L_m,H_m)$ where $m\geq 2$ is an integer, and $H_m$ is a certain subspace of the holomorphic functions on $\mathbb{H}$ \cite{Ro}. 
Furthermore, for the modular group $\Gamma=SL(2,\mathbb{Z})$, the classical cusp forms are also holomorphic functions with some $\Gamma$-invariant properties. 
V. Jones found that the multiplication by a cusp form of weight $p$ is in $B(H_m,H_{m+p})$ that intertwines the actions of $SL(2,\mathbb{Z})$ on $H_m$ and $H_{m+p}$. 
More precisely, the multiplication operator
\begin{center}
$M_f\colon H_m\to H_{m+p}$,  $\phi(z)\mapsto f(z)\phi(z)$
\end{center}
intertwines the actions of $\Gamma$ on $H_m$ and $H_{m+p}$, i.e., $M_fL_m(\gamma)=L_{m+p}(\gamma)M_f$ for all $\gamma \in SL(2,\mathbb{Z})$ \cite{GHJ}. 

Observe $PSL(2,\mathbb{Z})$ is an ICC group. 
(Recall that a group $G$ is an infinite conjugacy classes group, or an ICC group for short if every non-trivial conjugacy class $C_h=\{g^{-1}hg|g\in G\},h\neq 1$ is infinite.) 
Its group von Neumann algebra and the commutant are both factors of type $\text{II}_1$ (provided the formal dimension is finite). 
A natural question arises whether these operators composed with their adjoints, i.e., operators of the form $M_g^*M_f$,  generate the commutant factor. 
(Note the adjoint $M_f^*$ is more complicated than a single multiplication, see Section 4.3 or \cite{GHJ}.)
In 1994, F. Radulescu gave an affirmative answer by applying the Berezin quantization \cite{Ra94,Ra98}.
\begin{theorem}\label{tRa}[F. Radulescu, 1994]
The von Neumann algebras generated by the forms $M_g^*M_f$ is the commutant of the $\text{II}_1$ factor $L_m(PSL(2,\mathbb{Z}))''$, i.e.,
\begin{center}
$\overline{\{\text{span}_{f,g} M_g^*M_f\}}^{\text{w.o.}}=L_m(PSL(2,\mathbb{Z}))'$,   
\end{center}
where $f,g$ run through the cusp forms of same weights. 
\end{theorem}
But no result is known for other Fuchsian subgroups of $SL(2,\mathbb{R})$ or, more generally, lattices of a general real Lie group.   

In the first part of this paper, we generalize the result for $SL(2,\mathbb{Z})$ to the holomorphic discrete series of non-compact semi-simple real Lie groups. 
These representations can be denoted by $(L_{\pi},H_{\pi})$, where $H_{\pi}=L^2_{\text{holo}}(\mathcal{D},V_{\pi})$ and  $(\pi,V_{\pi})$ is an irreducible representation of $K$. 
We first use a generalized Berezin quantization to transfer each operator $A\in B(H_\pi)$ to an $\End(V_{\pi})$-valued function $S(A)(z)$ on $\mathcal{D}$ with some holomorphic properties (see Section \ref{sBeretr}). 
Once a discrete subgroup $\Gamma$ of the Lie group $G$ is given, we also give an explicit formula of a faithful normal tracial state on the commutant  $L_{\pi}(\Gamma)'=B(H_{\pi})^{\Gamma}=\{A\in B(H_{\pi})|AL_{\pi}(\gamma)=L_{\pi}(\gamma)A,\forall \gamma\in \Gamma\}$ of the group von Neumann algebra $\overline{L_{\pi}(\Gamma)}^{\text{s.o.}}$. 
Note $L_{\pi}(\Gamma)'$ is a type $\text{II}_1$ factor if $\Gamma$ is ICC. 

\begin{proposition}[The trace formula for $L_{\pi}(\Gamma)'$]
Assume $\pi$ is an irreducible representation of $K$.  
Let $\tau\colon B(H_{\pi})\to \mathbb{C}$ be the linear functional defined by
\begin{center}
$\tau(A)=\frac{1}{\mu(\mathcal{F})}\int_{\mathcal{F}}\tr(S(A)(z))d\mu(z)$,~~~~~ $A\in B(H_{\pi})$. 
\end{center}
Then $\tau$ is a positive, faithful, normal, normalized trace on $L_{\pi}(\Gamma)'$. 
In particular, if $\Gamma$ is an ICC group,  
$\tau$ is the unique normalized trace on the $\text{II}_1$ factor $L_{\pi}(\Gamma)'$. 
\end{proposition}

We then generalize the classical Toeplitz operator $T_f\in B(H_{\pi})$ associated with $f\in L^{\infty}(\mathcal{D})$ to an essentially bounded matrix Toeplitz operator $T_f$ associated with a  matrix-valued function $f$ on $\mathcal{D}$. 
Then the $\Gamma$-invariant functions can be identified with the ones defined on the fundamental domain $\Gamma\backslash\mathcal{D}$. 
Using several formulas of the tracial state of these operators, we prove
\begin{theorem}[Main Theorem I]
The commutant $L_\pi(\Gamma)'$ is generated by the Toeplitz operators of matrix-valued functions, i.e., 
\begin{center}
$\overline{\langle T_f|f\in L^{\infty}(\Gamma\backslash\mathcal{D},\End(V_{\pi}))\rangle}^{\text{w.o.}}=L_{\pi}(\Gamma)'$.
\end{center}
\end{theorem}
 
In the second part, we consider cusp forms defined on real Lie groups, which were first studied by Harish-Chandra \cite{HCa}. By definition, given a semi-simple real Lie group $G$, an automorphic form on $G$ is a complex (or complex vector-valued) function $f\colon G\to \mathbb{C}$ (or taking values in a finite-dimensional representation $V_{\rho}$ of $K$) which is $K$-right-finite (or right-equivariant),  $\Gamma$-left-invariant and satisfies some analytic properties.  
Indeed, we focus on another type of automorphic form defined on the domain $\mathcal{D}=G/K$, which can be easily obtained from the ones defined on the group $G$.  
As for intertwining properties of the classical cusp forms of the modular group $SL(2,\mathbb{Z})$, we also show the existence of $\Gamma$-invariant bounded linear operators between these holomorphic function spaces, which come from the cuspidal automorphic forms, or simply cusp forms, on general real Lie groups.  
Let $f\colon G(\text{or~}\mathcal{D}) \to V_{\rho}$ be a cusp form for $\Gamma$ of type $(\rho,V_{\rho})$ (here $(\rho,V_{\rho})$ is a  representation of $K$),  
which is not always holomorphic as in the case of $SL(2,
\mathbb{R})$. 
The multiplication operator $M_f$ is no longer closed. 
We construct a Toeplitz-type operator
\begin{center}
$T_f\colon H_{\pi}\to H_{\rho\otimes\pi}$ given by $\phi(z)\mapsto P_{\rho\otimes\pi}(f\otimes \phi)(z)$,
\end{center}
where $P_{\rho\otimes\pi}$ is the projection from $L^2(\mathcal{D},V_{\rho\otimes\pi})$ to the closed subspace $H_{\rho\otimes\pi}=L_{\text{hol}}^2(\mathcal{D},V_{\rho\otimes\pi})$ which is square-integrable and not always irreducible. 
Then $T_f$ also commutes with the actions of $\Gamma$ on $H_\pi$ and $H_{\rho\otimes\pi}$ respectively. 
This implies $T_f\in B(H_{\pi},H_{\rho\otimes\pi})^{\Gamma}$ and $T_g^* T_f\in L_{\pi}(\Gamma)'$ if $f,g$ are cusp forms of the same type. 
Our construction includes the earlier result on $SL(2,\mathbb{Z})\subset SL(2,\mathbb{R})$ as a special case  \cite{GHJ,Ra98,Ra14,J20}.

In this paper, Theorem \ref{tRa} is generalized to 
\begin{enumerate}
    \item Fuchsian subgroups of the first kind of $SL(2,\mathbb{R})$, 
    \item Lattices of real Lie groups (with holomorphic discrete series). 
\end{enumerate} 
The result on $SL(2,\mathbb{R})$ is obtained by proving certain existence theorems of meromorphic functions and holomorphic functions on the compact Riemann surface $\mathcal{F}^*=\Gamma/\mathbb{H}^*$. 
We apply Riemann-Roch theory for the proofs about meromorphic and holomorphic functions on $\mathcal{F}^*$. 
We prove there are enough cusp forms that can separate the points in the fundamental domain $\Gamma/\mathbb{H}$ of any Fuchsian group $\Gamma$ of the first kind, i.e., $\Gamma$ is a lattice. 

For the most general case, we apply Baily-Borel compactification and show the Poincar\'{e} series are abundant to separate points in the fundamental domain $\mathcal{F}=\Gamma\backslash \mathcal{D}$. 
We always assume $G$ has no normal $\mathbb{Q}$-subgroup of dimension $3$ \cite{BB}. 
Finally, we prove the following theorem in Section \ref{sauto} (see Theorem \ref{tmain}). 
\begin{theorem}[Main Theorem II]
The commutant $L_\pi(\Gamma)'$ can be generated by the cusp forms as following: 
\begin{center}
$\overline{\langle\{\text{span}_{f,g} T_g^{*}T_f\}\otimes \End(V_{\pi})\rangle}^{\text{w.o.}}=L_\pi(\Gamma)'$,
\end{center}
where $f,g$ run through holomorphic cusp forms for $\Gamma$ of same types. 
Moreover, if $\dim_{\mathbb{C}}V_{\pi}=1$, we have 
\begin{center}
$\overline{\langle \text{span}_{f,g} T_g^{*}T_f\rangle}^{\text{w.o.}}=L_\pi(\Gamma)'$, 
\end{center}
as $f,g$ run through holomorphic cusp forms for $\Gamma$ of same types.
\end{theorem}

Section 2 provides a brief discussion of the holomorphic discrete series representations and their realizations.  
Section 3 is devoted to the theory of the Berezin transform and construction of the matrix Toeplitz operators. 
We provide formulas for a trace $\tau$ on the finite von Neumann algebra $L_{\pi}(\Gamma)'$. 
In Section 4, we consider the extension of the Berezin transform of from $L_{\pi}(\Gamma)'$ to the standard module $L^2(L_{\pi}(\Gamma)',\tau)$ of it.  
In Section 5, we give an example on $SL(2,\mathbb{R})$ and its Fuchsian subgroups.  
In Section 6, we consider a real Lie group and construct $\Gamma$-intertwining operators from the cusp forms defined on the Lie group. 
Then the result in Section 5 is generalized to Theorem \ref{tmain}.

I wish to express my hearty thanks to my advisors Vaughan Jones and Dietmar Bisch. 
V. Jones proposed first to study the von Neumann algebras associated with cusp forms of $SL(2,\mathbb{Z})$ and encouraged me to work on it and its generalization. 
D. Bisch gave me many valuable suggestions and also the proofs involving the tracial von Neumann algebras in Section 4.2. 
My thanks go to Florin Radulescu, who discussed his early work with me and gave me some feedback.  
My thanks are also due to Larry Rolen for advice on the case $SL(2,\mathbb{Z})$, which I generalized in Section 5.3.  
This work could have never been done without them. 

\section{The Holomorphic Discrete Series}\label{shds}

We review some basic facts about discrete series representations. 
Then we focus on the family of holomorphic discrete series representations and their construction. 
We refer to \cite{HC,Kn,Neeb} for more details.  
\subsection{The discrete series representations}

Let $G$ be a locally compact unimodular group with Haar measure $dg$. 
We are interested in the case that $G=G_{\mathbb{R}}$ is a connected semi-simple real Lie group.
We assume $K$ is a maximal compact subgroup of $G$ and $H$ is the Cartan subgroup of $G$. 
We will use the following notations. 
\begin{itemize}
    \item $\mathfrak{h,k,g}$: the Lie algebra of $H,K,G$ respectively and $\mathfrak{h_{\mathbb{C}},k_{\mathbb{C}},g_{\mathbb{C}}}$ are their complexifications;
    \item $\Delta,\Delta_K$: roots of $(\mathfrak{g}_{\mathbb{C}},\mathfrak{h}_{\mathbb{C}})$ and $(\mathfrak{k}_{\mathbb{C}},\mathfrak{h}_{\mathbb{C}})$;
    \item $W_G,W_K$: the Weyl groups of $\Delta,\Delta_K$;
    \item $\delta_G,\delta_K$: the respective half-sums of positive roots.
\end{itemize}

Furthermore, we have in mind $G$ should be a non-compact group though we do not exclude the compact case. 
Let $\pi\colon G\to U(H)$ be a unitary representation of $G$ where $H$ is a Hilbert space with inner product $\langle \cdot,\cdot\rangle_H$.
For vectors $u,v\in H$, one defines the coefficient
\begin{center}
$g\in G \mapsto c_{u,v}(g)=\langle \pi(g)u,v\rangle_H$.
\end{center}
We can show $c_{u,v}(h^{-1}g)=c_{u,\pi(h)v}(g)$ and $c_{u,v}(gh)=c_{\pi(h)u,v}(g)$ for all $g,h\in G$. 
\begin{definition}
Let $\pi$ be a unitary representation of $G$. We say it is square-integrable if it has a non-zero square-integrable coefficient
\begin{center}
$0\neq c_{u,v}\in L^2(G,dg)$ for some $u,v\in H$.
\end{center}
If $\pi$ is irreducible, we call it a discrete series representation of $G$. 
\end{definition}
\begin{theorem}\cite{Ro}
Let $\pi$ be a unitary irreducible representation of a locally compact group $G$. 
The following properties are equivalent: 
\begin{enumerate}
\item There exist $u,v\in H$ such that $c_{u,v}$ is square-integrable.
\item For any $u,v\in H$, $c_{u,v}$ is square-integrable.
\item $\pi$ is equivalent to a subrepresentation of the right regular representation $\rho\colon G\to U(L^2(G,dg))$.
\end{enumerate}
\end{theorem}

For each discrete series representation $\pi\colon  G\to U(H)$, there is a parameter called {\it formal dimension} $d_\pi\in \mathbb{R}^{+}$ determined only by $\pi$, which is given by the following theorem.

\begin{theorem}[\cite{Ro}]
Let $(\pi,H)$ be a discrete series representation of $G$. 
Then there is a constant $d_\pi\in \mathbb{R}_{\geq 0}$ such that
\begin{center}
$\langle c_{u,v},c_{x,y}\rangle_{L^2(G)}=d_{\pi}^{-1}\langle u,x\rangle_H\cdot \overline{\langle v,y\rangle_H}$, for all $u,v,x,y\in H$.
\end{center}
Moreover, if $(\pi,H),(\pi',H')$ are two discrete series representations that are not equivalent, then $\langle c_{u,v},c_{u',v'}\rangle_{L^2(G)}=0$, for all $u,v\in H$ and $u',v'\in H'$. 
\end{theorem}

There is a criterion for the existence of discrete series representations proposed by Harish-Chandra and also proved by him. 
As in the case of the highest weight representations of compact Lie groups, these discrete series representations (up to unitary equivalence) can also be classified by their weights. 

\begin{theorem}[Harish-Chandra \cite{HC}]\label{tHCds}
The discrete series  representations exist if and only if $\rk G=\rk K$. 
They are classified by $\pi_{\lambda}$ with non-singular weight $\lambda \in (i\mathfrak{h})^{'}$ such that $\lambda+\delta_G$ is analytically integral. 
Moreover,  $\pi_{\lambda}\cong\pi_{\lambda'}$ if and only if $\lambda,\lambda'$ are conjugate under $W_K$. 
\end{theorem}
Here $(i\mathfrak{h})^{'}$ denotes the dual space of $i\mathfrak{h}$.

Note when $G$ is compact, i.e., $G=K$, this theorem reduces to the theorem of highest weight representations. 
In particular, for a complex Lie group $G_{\mathbb{C}}$ with a compact real form $G_{\mathbb{R}}$, we have $\rk G_{\mathbb{C}}=2\rk G_{\mathbb{R}}$, and it never has discrete series representations. 
The classical groups $SL(2,\mathbb{R}),SU(p,q),Sp(n,\mathbb{R})$ have holomorphic discrete series. 
More details of the construction of these representations can be found in \cite{Kn}. 
A geometric realization of these representations or the generalized Borel-Weil-Bott theorem using  $L^2$-cohomology was conjectured by R. Langlands and then proved by W. Schmid \cite{Sch}. 

\subsection{Construction of the holomorphic discrete series}\label{sholds}

The holomorphic discrete series are the discrete series that can be represented in a natural way by   Hilbert spaces of holomorphic functions. 
We refer to \cite{HC,Neeb,Oli} for the relevant descriptions.  
From now on, we always assume $G$ is a connected non-compact semi-simple real Lie group with $\rk G=\rk K$, and $K$ has a non-finite center.  

Let $\theta\in \Aut(\mathfrak{g})$ be a Cartan involution and $\mathfrak{g}=\mathfrak{k}\oplus \mathfrak{p}$ be the Cartan decomposition.  
Then we have $\mathfrak{k}\oplus i\mathfrak{p}$ is a compact real form of $\mathfrak{g}_{\mathbb{C}}$. 
We also write $Z^*=(X+iY)^*=-\theta(X+iY)=-X+iY$ for $Z=X+iY\in \mathfrak{g}_{\mathbb{C}}$ with $X,Y\in \mathfrak{g}$. 
For $g=\exp(X+iY)\in G_\mathbb{C}$, we write $g^*=\exp((X+iY)^*)=\exp(-X+iY)=\overline{g^{-1}}$. 
(The notation $Z^*$ and $g^*$ has this meaning only when we consider Lie algebras and Lie groups)

Consider the root space decomposition  $\mathfrak{g}_{\mathbb{C}}=\mathfrak{h}_{\mathbb{C}}\oplus \sum_{\alpha \in \Delta}\mathfrak{g}_{\alpha}$.
We further obtain $\mathfrak{k}_{\mathbb{C}}=\mathfrak{h}_{\mathbb{C}}\oplus \sum_{\alpha \in \Delta_K}\mathfrak{g}_{\alpha}$ and 
$\mathfrak{p}_{\mathbb{C}}= \sum_{\alpha \in \Delta_n}\mathfrak{g}_{\alpha}$ where $\mathfrak{p}_{\mathbb{C}}$ is the complexification of $\mathfrak{p}$ and $\Delta_n=\Delta-\Delta_K$ is the set of noncompact roots. 
Let $\Delta^+,\Delta_K^+,\Delta_n^+$ be selected sets of positive roots, positive compact roots, and positive non-compact roots, respectively. 
We set $\mathfrak{p}^+=\sum_{\alpha \in \Delta_n^+}\mathfrak{g}_{\alpha}$ and $\mathfrak{p}^-=\sum_{\alpha \in \Delta_n^-}\mathfrak{g}_{\alpha}$

Let $P^+$ and $P^-$ be the analytic subgroups of $G_{\mathbb{C}}$ with Lie algebras $\mathfrak{p}^+$ and $\mathfrak{p}^-$ respectively. 
There is a diffeomorphism $(z^+,k,z^-)\mapsto z^{+}\cdot k\cdot z^-$ from $P^+\times K_{\mathbb{C}}\times K^-$ to an open submanifold of $G_{\mathbb{C}}$ containing $G$ \cite{Hel}.  
Following \cite{Neeb}, we also introduce three projections
\begin{center}
$\zeta\colon P^{+}K_{\mathbb{C}}P^{-}\to P^{+}$, $\kappa\colon P^{+}K_{\mathbb{C}}P^{-}\to K_{\mathbb{C}}$, and $\xi\colon P^{+}K_{\mathbb{C}}P^{-}\to P^{-}$.     
\end{center}
Then the map $\phi\colon G/K\to \mathfrak{p}^+$ given by
$\phi(gK)=\log(\zeta(g))$ induces a diffeomorphism from $G/K$ to a bounded domain $\mathcal{D}\subset \mathfrak{p}^+$. 

The domain $\mathcal{D}$ is an irreducible Hermitian symmetric space of non-compact type \cite{Hel,Oli} 
and is known as the {\it Harish-Chandra realization} of $G/K$. 
Let $\mu$ be the $G$-invariant measure on $\mathcal{D}$ given by
\begin{center}
$d\mu(z)=\det\ad_{\mathfrak{p}^+}\kappa^{-1}(z,z)dz$,    
\end{center} 
and $dz$ is the Euclidean measure on $\mathcal{D}\subset \mathfrak{p^+}$ \cite{Sa}. 
We will identify $\mathcal{D}$ with $G/K$
in the following sections. 

Let $B(\cdot,\cdot)$ be the Killing form on $\mathfrak{g}_{\mathbb{C}}$.
We have
\begin{theorem}[\cite{Sa}]\label{tbddo}
The bounded symmetric domain can be given by
\begin{center}
$\mathcal{D}=\{z\in\mathfrak{p}^+|\|\ad z\|<\sqrt{2}\}$
\end{center}
where the norm is the operator norm on $\mathfrak{g}_{\mathbb{C}}$ equipped with the positive definite Hermitian form $-B(X,\theta\overline{Y})$. 
\end{theorem}

Recall the map $\kappa\colon P^{+}K_{\mathbb{C}}P^{-}\to K_{\mathbb{C}}$ defined above.
We define a map from $\mathcal{D}\times \mathcal{D}$ to $K_{\mathbb{C}}$, which is also denoted by $\kappa$, by
\begin{center}
 $\kappa(z,w)=\kappa(\exp (w^{*})\exp(z))^{-1}$, $z,w\in\mathcal{D}$.     
\end{center}
We define a map $J\colon G\times \mathcal{D}\to K_{\mathbb{C}}$ by
\begin{center}
$J\colon G\times \mathcal{D}\to K_{\mathbb{C}}$ by $J(g,z)=\kappa(g\exp z)$, $g\in G,z\in \mathcal{D}$.    
\end{center}
The map $J$ is usually called the {\it canonical automorphy factor} of $G$,  which satisfies the following properties:  
\begin{enumerate}[label=(\roman*)]
    \item $J(g,z)$ is $C^{\infty}$ in the first variable and holomorphic in the second one,
    \item $J$ is a 1-cocycle, i.e., $J(gh,w)=J(g,hw)J(h,w)$ for $g,h\in G$, $w\in \mathcal{D}$, 
    \item $J(k,0)=k$ if $k\in K$. 
\end{enumerate}
For $\kappa$, we have an alternative definition which is determined by the following properties: (i)$\kappa(0,0)=e$, (ii) $\kappa(z,w)$ is holomorphic in $z$, (iii) $\kappa(z,w)=\kappa(w,z)^*$ and (iv)$\kappa(gz,gw)=J(g,z)\kappa(z,w)J(g,w)^*$. 

\begin{remark}
\label{r1}
Let $\rho$ be a unitary representation of $K$ (so it can be extended to $K_{\mathbb{C}}$). 
For $g\in K_{\mathbb{C}}$, we have $\rho(g^*)=\rho(g)^*$, the adjoint operator associated with $\rho(g)$. 

Indeed, we assume $\rho(g)=\exp(X+iY)$ with $X,Y\in \rho(\mathfrak{k}_{\mathbb{R}})$. 
Then $\rho(g^*)=\exp(-X+iY)$. 
As $\rho$ is unitary on $K$, $X+X^*=0$. 
So we obtain $\exp(-X+iY)=\exp(X^*-i\cdot Y^*)=\exp((X+iY)^*)=\rho(g)^*$. 
\end{remark}

Now we can construct the holomorphic discrete series of $G$. 
Let $(\pi,V_{\pi})$ be a finite dimensional unitary representation of the compact subgroup $K$ and  $\langle,\rangle_{\pi}\colon V_{\pi}\times V_{\pi}\to \mathbb{C}$ is a $K$-invariant inner product.  
We also let $(\pi,V_{\pi})$ denote the representation extended to $K_{\mathbb{C}}$. 
Here we do not assume the irreducibility of $\pi$. 

Consider the space
\begin{center}
$\text{Map}(\mathcal{D},V_{\pi})=\{f\colon \mathcal{D}\to V_{\pi}|f\text{~is~measurable}\}$    
\end{center}
and the following inner product on it: 
\begin{center}
$\langle f,h\rangle=\int_{\mathcal{D}}\langle \pi(\kappa(z,z)^{-1})f(z),h(z)\rangle_{\pi}d\mu(z)$, $f,h\in \text{Map}(\mathcal{D},V_{\pi})$. 
\end{center}
 
Following Remark \ref{r1} above, one can check it is positive-definite. 
We define
\begin{center}
$L^2(\mathcal{D},V_{\pi})=\{f\in\text{Map}(\mathcal{D},V_{\pi})|\langle f,f\rangle<\infty \}$,
\end{center}
and also the subspace spanned by holomorphic functions
\begin{center}
$H_{\pi}=L_{\text{hol}}^2(\mathcal{D},V_{\pi})=\{f\in L^2(\mathcal{D},V_{\pi})| f\text{~is~holomorphic} \}$    
\end{center}
where the inner product restricted to $H_\pi$ will be written as $\langle,\rangle_{H_{\pi}}$. 
It can be shown that $H_{\pi}$ is closed subspace of $L^2(\mathcal{D},V_{\pi})$.  
The action $L_{\pi}$ of $G$ on $H_{\pi}$ is given by
\begin{center}
$L_{\pi}(g)f(z)=\pi(J(g^{-1},z)^{-1})f(g^{-1}z)$, $f\in H_{\pi},g\in G,z\in \mathcal{D}$. 
\end{center}
It can be proved that this representation $(L_{\pi},H_{\pi})$ is a unitary representation of $G$ \cite{Oli}. 

Please note the unitary representations $(L_{\pi},H_{\pi})$ are square-integrable, i.e., $H_{\pi}\subset L^2(G)$. 
The {\it holomorphic discrete series} of $G$ are the $(L_{\pi},H_{\pi})$'s, where $\pi$ is an irreducible representation of $K$.
\begin{theorem}[\cite{HC,Kn}]\label{tirreKG}
Assume $\pi$ is an irreducible representation of $K$ with highest weight $\Lambda$, then $H_{\pi}$ is nonzero if and only if $(\Lambda+\delta_G)(H_\beta)<0$ for all $\beta\in \Delta_n^+$. 
In this case, it is irreducible. 
\end{theorem}

The formal dimensions of these representations can be given by explicit formulas \cite{Neeb,Oli}. 
Now we let $P_{\pi}$ be the orthogonal projection from 
$L^2(\mathcal{D},V_{\pi})$ to $H_{\pi}$. 
Let $L_{\pi}$ be the action of $G$ on $L^2(\mathcal{D},V_{\pi})$ defined by
\begin{center}
 $L_{\pi}(g)f(z)=\pi(J(g^{-1},z)^{-1})f(g^{-1}z)$, $f\in H_{\pi},g\in G,z\in \mathcal{D}$, $f\in L^2(\mathcal{D},V_{\pi})$. 
\end{center}
This is to say $H_{\pi}$ is a $G$-invariant subspace and $(L_{\pi},L^2(\mathcal{D},V_{\pi}))$ also gives us a well-defined unitary representation. 
\begin{proposition}
\label{pinvariant}
For $g\in G$, $L_{\pi}(g)$ and $P_{\pi}$ commute on $L^2(\mathcal{D},V_{\pi})$, i.e., $L_{\pi}(g)P_{\pi}f=P_{\pi}L_{\pi}(g)(f)$ for $f\in L^2(\mathcal{D},V_{\pi})$. 
\end{proposition}
\begin{proof}
Take $h\in L^2(\mathcal{D},V_{\pi})$. Assume $h=h_0\oplus h_1$ with $h_0=P_{\pi}(h),h_1=(1-P_{\pi})(h)$ and also let $f_1\in (1-P_{\pi})L^2(\mathcal{D},V_{\pi})$.  
We can show $L_{\pi}$ leaves $H_{\pi}^{\perp}$ invariant: $\langle L_{\pi}(g)f_1,h_0\rangle_{L^2}=\langle f_1,L_{\pi}(g^{-1})h_0\rangle=0$. 
Hence $L_{\pi}(g)P_{\pi}f=P_{\pi}L_{\pi}(g)f$. 
\end{proof}

\section{Berezin Transform, Toeplitz Operators and Trace Formulas}

In this section, we will study the holomorphic discrete series representations $(L_{\pi},H_{\pi})$ of a semi-simple real Lie group $G$ with restriction to a lattice $\Gamma\subset G$ and the related von Neumann algebras. 
We construct some Berezin quantizations and generalized Toeplitz operators, emphasizing the $\Gamma$-invariant properties. 
We obtain several explicit formulas of a trace on the commutant. 

\subsection{Berezin symbols and a trace}\label{sBeretr}

The Berezin quantization is defined for the upper half-plane $\mathbb{H}=\{z=x+iy\in \mathbb{C}|y>0\}$.
It may also be defined for the open unit disk $\mathbb{D}=\{z\in \mathbb{C}||z|<1\}$ by the Cayley transform, which sends $z\in \mathbb{H}$ to $\frac{z-i}{z+i}\in \mathbb{D}$ \cite{Be74,Be75}. 
This is a special case of the bounded symmetric domain $\mathcal{D}=G/K$ when $G=SL(2,\mathbb{R})$ and $K =SO(2)$. 
We will focus on its generalization to the bounded symmetric domain $\mathcal{D}=G/K$ for a Lie group $G$ with holomorphic discrete series as introduced in the previous sections.

Let $(L_{\pi}, H_{\pi})$ be a nontrivial square-integrable representation of $G$ which is associated to the finite dimensional unitary representation $(\pi,V_{\pi})$ of a maximal compact subgroup $K$. 
Recall that $H_{\pi}=L^2_{\text{holo}}(\mathcal{D},V_{\pi})$. 
Here we do not assume it is irreducible (and $(\pi,V_{\pi})$ neither by Theorem \ref{tirreKG}).  
For any $z\in D$, the evaluation function $K_z\colon H_{\pi}\to V_{\pi}$ given by $f\mapsto f(z)$ is continuous and bounded \cite{Neeb}. 
Hence its adjoint operator $E_{z}=K_z^*\colon V_{\pi} \to H_{\pi}$ is defined by
\begin{center}
$\langle K_zf,v\rangle_{\pi}=\langle f(z),v\rangle_{\pi}=\langle f,E_z(v)\rangle_{H_\pi}$, $\forall f\in H_{\pi},v\in V_{\pi}$, 
\end{center}
where $\langle \cdot,\cdot\rangle_{\pi}$ is the $K$-invariant inner product on $V_{\pi}$. 

Let $G=NAK$ be the Iwasawa decomposition. 
There is a smooth embedding
\begin{center}
$i\colon \mathcal{D}\cong G/K\hookrightarrow NA\subset G$, $z\mapsto g_z$.   
\end{center} 
Please note we have $z=\dot{g_z}=g_z\cdot K$ as a coset in $G/K$.
We denote $h_z=\kappa(g_z)\in K_{\mathbb{C}}$ for $z\in \mathcal{D}$ and also let $H_z=\pi(h_z^{-1})\in GL(V_{\pi})$ (see Section \ref{sholds} for the map $\kappa\colon P^{+}K_{\mathbb{C}}P^{-}\to K_{\mathbb{C}}$).  

The following result and its proof can be found in \cite{Neeb,Cah}. 
\begin{lemma}
\label{lkappa}
\begin{enumerate}[label=(\roman*)]
    \item There exists a constant $c_{\pi}\in \mathbb{R}_{>0}$ such that $E_z^*E_w=c_{\pi}\pi(\kappa(z,w))$ for $z,w\in \mathcal{D}$.
    \item $E_{gz}=L_{\pi}(g)E_z\pi(J(g,z)^*)$ for $g\in G,z\in \mathcal{D}$.
\end{enumerate}
\end{lemma}

Now we define four {\it Berezin symbols} for operators in $B(H_\pi)$. 
Recall $c_{\pi}\in\mathbb{R}$ is the contant given in Lemma \ref{lkappa}. 
\begin{definition}
For an operator $A\in B(H_\pi)$, the Berezin symbols of $A$ are defined as the following $\End(V_{\pi})$-valued function: 
\begin{enumerate}
\item $K_A(z,w)=E^*_{z}AE_w$, 
\item $R(A)(z,w)=\frac{1}{c_{\pi}}H_wK_A(w,z)H_z^*=\frac{1}{c_{\pi}}H_wE_w^*AE_zH_z^*$,
\item $S(A)(z)=R(A)(z,z)=\frac{1}{c_{\pi}}H_zK_A(z,z)H_z^*=\frac{1}{c_{\pi}}H_zE_z^*AE_zH_z^*$,
\item $Q(A)(z)=\frac{1}{c_\pi}K_A(z,z)H_z^*H_z=\frac{1}{c_\pi}E_z^*AE_zH_z^*H_z$
\end{enumerate}
where $z,w\in \mathcal{D}$.
\end{definition}

Note $K_A,R(A)$ are maps from $\mathcal{D}\times \mathcal{D}$ to $\End(V_\pi)$ and $S(A),Q(A)$ are maps from $\mathcal{D}$ to $\End(V_\pi)$.

Let $U$ be a non-empty open subset of $ \mathbb{C}^N\times \mathbb{C}^N$. 
Given a function $f(z,w)\colon U\to \mathbb{C}$, we call it {\it sesqui-holomorphic} if $f$ is holomorphic in both $z$ and $\bar{w}$. 

\begin{theorem}[\cite{BM48} II.4.]\label{tsesholo}
Assume a complex function $f(z,w)$ of $2N$ complex variables $z=z_1,\dots,z_N$ and $w=w_1,\dots,w_N$ is given in a neighborhood of the origin $(0,0)$. 
If $f$ is sesqui-holomorphic function and $f(z,\bar{z})=0$ for all $z$, 
then we have $f=0$. 
\end{theorem}

Now we fix the bounded domain $\mathcal{D}$ with the measure $\mu$ given in Section \ref{sholds}.
We give some properties of the Berezin symbols. 
The result concerning only $K_A$ can be found in \cite{Cah}.  
\begin{proposition}
\label{pK}
\begin{enumerate}[label=(\roman*)]
    \item $K_{A^*}(z,w)=K_A(w,z)^*$.
    \item $K_A(z,w)$ is sesqui-holomorphic. 
    \item The correspondences $A\mapsto K_A(z,w)$, $A\mapsto K_A(z,z)$ and $K_A(z,w)\mapsto K_A(z,z)$ are all injective.    
    \item $K_A(gz,gw)=\pi(J(g,z))K_{L_{\pi}(g)^{-1}AL_{\pi}(g)}(z,w)\pi(J(g,w))^*$. 
\end{enumerate}
\end{proposition}
\begin{proof}
For (i), we have $K_A(w,z)^*=(E_w^* A E_z)^*=E_z^* A^* E_w=K_{A^*}(z,w)$. 

For (ii), observe $K_A(z,w)(v)=(AE_w(v))(z)$ is a holomorphic function of $z$ as it belongs to $H_{\pi}$. 
Furthermore, by (i), it is anti-holomorphic in $w$. 

For (iii), we first show $ K_A(z,w)$ determines $A$.
Take $\phi\in H_{\pi}$ and consider the inner product $\langle (A\phi)(z),v\rangle_{\pi}$. 
\begin{equation*}
\begin{aligned}
\langle (A\phi)(z),v\rangle_{\pi}&=\langle \phi,A^* E_z(v)\rangle_{H_{\pi}}\\
&=\int_{\mathcal{D}}\langle\pi(\kappa(w,w)^{-1})E_w^*\phi,E_w^*A^* E_z(v)\rangle_{\pi}d\mu(w)\\
&=\int_{\mathcal{D}}\langle\pi(\kappa(w,w)^{-1})\phi(w),K_A(z,w)^*(v)\rangle_{\pi}d\mu(w) 
\end{aligned}
\end{equation*}
Hence $K_A(z,w)$ determines $A$. 
For the injectivity of the map $K_A(z,w)\mapsto K_A(z,z)$, we consider the $K_A(z,w)=[k_{i,j}(z,w)]_{n\times n}$ (here $n=\dim V_{\pi}$) such that each $k_{i,j}(z,w)\colon \mathcal{D}\times \mathcal{D}\to \mathbb{C}$ is a sesqui-holomorphic function. 
Then it immediately follows from Theorem \ref{tsesholo}. 

Moreover, (iv) is a direct consequence of (ii) of Lemma \ref{lkappa}.  
\end{proof}

The Berezin symbol $S(A)$ also has some similar properties. 
\begin{proposition}
\label{pS}
Given $A\in B(H_{\pi})$, we have: 
\begin{enumerate}[label=(\roman*)]
    \item The maps $A\mapsto K_A,R(A),S(A),Q(A)$ are all injective.
    \item $S(A)(z)^{*}=S(A^*)(z)$ and $R(A)(z,w)^*=R(A^*)(w,z)$ for $z,w\in\mathcal{D}$. 
    \item $S(I)(z)=I_{V_{\pi}}$ for each $z\in \mathcal{D}$.
    \item The spectral radius of $S(A)(z)$ on $V_{\pi}$ is bounded by $\|A\|$ for each $z\in\mathcal{D}$. 
    \item For $g\in G,z\in\mathcal{D}$, we have
    \begin{center}
        $S(A)(gz)=\pi(k(g,z)^{-1})S(L_{\pi}(g)^{-1}AL_{\pi}(g))(z)\pi(k(g,z))$
    \end{center}
    where $k(g,z)=h_z^{-1}\kappa(g\exp{z})^{-1}h_{gz}$. 
\end{enumerate}
\end{proposition}
\begin{proof}
(i) and (ii) follow Proposition \ref{pK}. 

For (iii), note $K_I(z,z)=E^*_{z}E_z=c_{\pi}\pi(\kappa(z,z)^{-1})$ by Lemma \ref{lkappa}(i), it suffices to prove $\pi(h_zh_z^*)=\pi(\kappa(z,z)^{-1})$. 
We write $g_z=\exp{z}\cdot h_{z}\cdot p$ with $p\in P^{-}$ (see the definition of $\kappa$ in Section 2.2). 
As $g_z\in G=G_{\mathbb{R}}$, we have $g_z^*=\overline{g_z^{-1}}=g_z^{-1}$ and $\pi(g_z^*)\pi(g_z)=I_v$.
Hence 
\begin{center}
$\pi(p)^{*}\pi(h_z)^{*}\pi(\exp{z})^*\pi(\exp{z})\pi(h_z)\pi(p)=I_V$.     
\end{center}
We further obtain
\begin{center}
$\pi(h_z)^{*}\pi(\exp{z})^*\pi(\exp{z})\pi(h_z)=I_V$,\\ and $\pi(h_z)^{*}\pi(\kappa(\exp{z^*}\exp{z}))\pi(h_z)=I_V$    
\end{center} 
where the middle term is just $\pi(\kappa(z,z)^{-1})$. 

For (iv), we assume $\lambda$ is an eigenvalue of $A$ of the maximal modulus and $v\neq 0$ is the corresponding eigenvector. 
Note $S(I)(z)=I_{V_{\pi}}$, we have 
\begin{equation*}
\begin{aligned}
|\lambda|&=\left|\frac{\langle S(A)(z)v,v\rangle_{\pi}}{\langle v,v\rangle_{\pi}}\right|=\left|\frac{\langle S(A)(z)v,v\rangle_{\pi}}{\langle S(I)(z)v,v\rangle_{\pi}}\right|=\left|\frac{c_{\pi}^{-1}\cdot\langle AE_zH_z^* v,E_zH_z^*v\rangle_{\pi}}{c_{\pi}^{-1}\cdot\langle E_zH_z^* v,E_zH_z^*v\rangle_{\pi}}\right|\\
&=\left|\frac{\langle AE_zH_z^* v,E_zH_z^*v\rangle_{\pi}}{\langle E_zH_z^* v,E_zH_z^*v\rangle_{\pi}}\right|\leq
\frac{\|AE_zH_z^* v\|_{\pi}\cdot \|E_zH_z^* v\|_{\pi}}{\|E_zH_z^*v\|_{\pi}^2}\leq \|A\|. 
\end{aligned}
\end{equation*}
So $\lambda\leq \|A\|$. 

For (v), we also assume $g_z=\exp{z}\cdot h_{z}\cdot p$. 
Following \cite{Hel}, the action $G\curvearrowright \mathcal{D}$ induced from $G\curvearrowright G/K$ is given as $g\cdot z=\log{\zeta(g\exp{z})}$. 
Hence we have
\begin{center}
$gg_z=g\exp{z}h_{z}y=\exp(g\cdot z)\kappa(g\exp{z})\xi(g\exp{z})h_{z}y$. 
\end{center}
We have $g_{g\cdot z}=\exp(g\cdot z)h_{g\cdot z}y'$ for some $y'\in P^-$. 
Note $\dot{gg_z}=\dot{g_{g\dot z}}$ in $G/K$, there is some $k=k(g,z)\in K$ such that $g_{g\cdot z}=gg_{z}k$. 
So we obtain $\kappa(g\exp{z})\xi(g\exp{z})h_{z}yk=h_{g\cdot z}y'$ and further $\kappa(g\exp{z})h_zk=h_{g\cdot z}$ by applying $\kappa$. 
Now we can apply Proposition \ref{pK}(iv) and obtain
\begin{equation*}
\begin{aligned}
S(A)(g\cdot z)&=\frac{1}{c_{\pi}}\pi(h_{g\cdot z}^{-1})\pi(\kappa(g\exp{z}))E_z^* L_{\pi}(g)^{-1}AL_{\pi}(g)\pi(\kappa(g\exp{z}))^*\pi(h_{g\cdot z}^{-1})^*\\
&=\frac{1}{c_{\pi}}\pi(k)^{-1}\pi(h_z^{-1})E_z^* L_{\pi}(g)^{-1}AL_{\pi}(g)\pi(h_z^{-1})^*\pi(k)\\
&=\pi(k)^{-1}S( L_{\pi}(g)^{-1}AL_{\pi}(g))(z)\pi(k)
\end{aligned}
\end{equation*}
\end{proof}

Parts (ii), (iii) and (v) of Proposition \ref{pS} are first proved by B. Cahen \cite{Cah}. 

Now we are able give an explicit formula for the projection $P_{\pi}$. 
\begin{lemma}\label{lproj}
Given any $\phi\in L^2(\mathcal{D},V_{\pi})$, its image under $P_{\pi}$ is given by
\begin{center}
$(P_{\pi}\phi)(z)=\int_{\mathcal{D}}E^*_{z}E_wH_w^*H_w\phi(w)d\mu(w)$.
\end{center}
\end{lemma}
\begin{proof}
As $S(I)(w)=I_{V_\pi}$, we have $\frac{1}{c_{\pi}}H_wE_w^*E_wzH_w^*=I_{V_\pi}$. 
Then, by Lemma \ref{lkappa}(i), we get
$\pi(\kappa(w,w)^{-1})=H^*_wH_w$. 
Now let $v\in V_{\pi},z\in\mathcal{D}$ and consider the inner product $\langle \phi,E_z(v)\rangle_{L^2(\mathcal{D},V_{\pi})}$. 
We have
\begin{equation*}
\begin{aligned}
\langle (P_{\pi}\phi)(z),v\rangle_{\pi}&=\langle E_Z^* (P_{\pi}\phi),v \rangle_{\pi}=
\langle P_{\pi}\phi,E_z{v}\rangle_{H_\pi}=\langle \phi,E_z{v}\rangle_{L^2}\\
&=\int_{\mathcal{D}}\langle \pi(\kappa(w,w)^{-1})\phi(w),E_w^* E_z v\rangle_{\pi}d\mu(w)\\
&=\int_{\mathcal{D}}\langle H_w\phi(w),H_wE_w^* E_z v\rangle_{\pi}d\mu(w)\\
&=\langle \int_{\mathcal{D}} E_z^*E_wH_w^*H_w\phi(w)d\mu(w),v\rangle_{\pi}
\end{aligned}
\end{equation*}
which completes the proof.  
\end{proof}

We denote by $\tr=\tr_{\pi}$ the normalized trace on $\End(V_{\pi})$. 
\begin{lemma}\label{linvtr}
Let $A\in B(H_{\pi})$ such that it commutes with the action of $\Gamma$, i.e., $AL_{\pi}(\gamma)=L_{\pi}(\gamma)A$ for any $\gamma\in \Gamma$. 
Then we have 
\begin{enumerate}
    \item $\tr(S(A)(z))$ is $\Gamma$-invariant,
    \item $\tr(R(A)(z,w)R(B)(z,w)^*)$ is $\Gamma$-invariant.
\end{enumerate}
\end{lemma}
\begin{proof}
The first statment follows from the fact $\tr_{\pi}(S(A)(gz))=\tr_{\pi}(S(L_{\pi}(g)^{-1}AL_{\pi}(g))(z))$ in Proposition \ref{pS} (v).  

For the second statement, let $\gamma\in \Gamma$ and $\tilde{z}\in G$ be the inverse image of $z\in \mathcal{D}$. 
Note that
\begin{equation*}
\begin{aligned}
&\pi(J(\gamma,z))^*H_{\gamma z}^*H_{\gamma z}\pi(J(\gamma,z))\\
=&\pi(J(\gamma,z))^* \pi(J(\tilde{\gamma z},0)^{-1})^*\pi(J(\tilde{\gamma z},0)^{-1})\pi(J(\gamma,z))\\
=&(\pi(J(\tilde{\gamma z},0)^{-1})\pi(J(\gamma,z)))^*\cdot (\pi(J(\tilde{\gamma z},0)^{-1})\pi(J(\gamma,z)))\\
=&\pi(J(\tilde{z},0)^{-1})^*\pi(J(\tilde{z},0)^{-1})=H_z^*H_z
\end{aligned}
\end{equation*}
Then we have
\begin{equation*}
\begin{aligned}
&\tr(R(A)(\gamma z,\gamma w)R(B)(\gamma z,\gamma w)^*)\\
=&\frac{1}{c_\pi^2}\tr(H_{\gamma w}K_A(\gamma w,\gamma z)H_{\gamma z}^*H_{\gamma z} K_{B^*}(\gamma z,\gamma w)H_{\gamma w}^*)\\
=&\frac{1}{c_\pi^2}\tr(H_{\gamma w}\pi(J(\gamma,w))K_A(w,z)\pi(J(\gamma,z))^*H_{\gamma z}^*H_{\gamma z}\pi(J(\gamma,z)) K_{B^*}(z,w)\pi(J(\gamma,w))^*H_{\gamma w}^*)\\
=&\frac{1}{c_\pi^2}\tr(K_A(w,z)\pi(J(\gamma,z))^*H_{\gamma z}^*H_{\gamma z}\pi(J(\gamma,z)) K_{B^*}(z,w)\pi(J(\gamma,w))^*H_{\gamma w}^*H_{\gamma w}\pi(J(\gamma,w)))\\
=&\frac{1}{c_\pi^2}\tr(
H_wK_A(w,z)H_z^*H_z K_{B^*}(z,w)H_w^*)=\tr(R(A)(z,w)R(B)(z,w)^*). 
\end{aligned}
\end{equation*}
\end{proof}

Let $\mathcal{F}=\Gamma\backslash\mathcal{D}$ be the fundamental domain of the left action of $\Gamma$ on $\mathcal{D}=G/K$. 
Suppose the discrete group $\Gamma$ is a lattice, we have $\mu(\mathcal{F})$ is finite. 
Now we define
\begin{center}
$A_{\pi}=B(H_{\pi})^{\Gamma}=\{A\in B(H_{\pi})|AL_{\pi}(\gamma)=L_{\pi}(\gamma)A,\forall \gamma\in \Gamma\}$,     
\end{center}
which is the commutant of the von Neumann algebra $L_{\pi}(\Gamma)''$.  

\begin{proposition}
\label{ptr1}
Assume $\pi$ is an irreducible representation of $K$.  
Let $\tau\colon B(H_{\pi})\to \mathbb{C}$ be the linear functional defined by
\begin{center}
$\tau(A)=\frac{1}{\mu(\mathcal{F})}\int_{\mathcal{F}}\tr(S(A)(z))d\mu(z)$, $A\in B(H_{\pi})$. 
\end{center}
Then $\tau$ is a positive, faithful, normal, normalized trace on $A_{\pi}$. 

In particular, if $\Gamma$ is an ICC group,  
$\tau$ is the unique normalized trace on the $\text{II}_1$ factor $A_{\pi}$. 
\end{proposition}
\begin{proof}
We first show $|\tau(A)|<\infty$ for all $A\in B(H_{\pi})$. 
By Proposition \ref{pS} (iv), we know  $|\tr(S(A)(z))|\leq \|A\|$ and the integral is finite as $\mu({\mathcal{F}})$ is finite. 
Note $\tau(I)=1$ and $\tau(A^*A)\geq 0$ since $S(A^*A)(z)=S(A^*A)(z)^*$ by Proposition \ref{pS} (v). 
So it is normalized and positive. 

Suppose $\dim_{\mathbb{C}}V_{\pi}=n$ and take an orthonormal basis $\{v_i,1\leq i\leq n\}$ of $V_{\pi}$. 
One has
\begin{small}
\begin{equation*}
\begin{aligned}
\tau(A^*A)=&\frac{1}{\mu(\mathcal{F})}\int_{\mathcal{F}}\tr(S(A^*A)(z))d\mu(z)\\
=&\frac{1}{\mu(\mathcal{F})}\int_{\mathcal{F}}\tr(\frac{1}{c_{\pi}}H_zE^*_zA^*AE_zH_z^*)d\mu(z)\\
=&\frac{1}{n\cdot c_{\pi}\cdot\mu(\mathcal{F})}\int_{\mathcal{F}}\sum_{i=1}^{n}\langle H_z{E^*_z}{A}^*{A}E_z{H_z^*}v_i,v_i\rangle_{\pi}d\mu(z)\\
=&\frac{1}{n\cdot c_{\pi}\cdot\mu(\mathcal{F})}\int_{\mathcal{F}}\sum_{i=1}^{n}\langle {A}E_{z}H_z^*v_i,{A}E_{z}H_z^*v_i\rangle_{H_{\pi}}d\mu(z)\\
=&\frac{1}{n\cdot c_{\pi}\cdot\mu(\mathcal{F})}\int_{\mathcal{F}}\sum_{i=1}^{n}\left( \int_{\mathcal{D}}\langle \pi(\kappa(w,w)^{-1})({A}E_{z}H_z^*v_i)(w),({A}E_{z}H_z^*v_i)(w)\rangle_{\pi}d\mu(w)\right)d\mu(z)\\
=&\frac{1}{n\cdot c_{\pi}\cdot\mu(\mathcal{F})}\int_{\mathcal{F}}\sum_{i=1}^{n}\left( \int_{\mathcal{D}}\langle {H_w}E_w^*{A}E_{z}H_z^*v_i,{H_w}E_w^*{A}E_{z}H_z^*v_i\rangle_{\pi}d\mu(w)\right)d\mu(z)
\\
=&\frac{1}{ c_{\pi}\cdot\mu(\mathcal{F})}\int_{\mathcal{F}}\left( \int_{\mathcal{D}}\tr((H_z{E_z^*}A^*{E_w}H_w^*) ({H_w}E_w^*{A}E_{z}H_z))d\mu(w)\right)d\mu(z)
\\
=&\frac{1}{ c_{\pi}\cdot\mu(\mathcal{F})}\int_{\mathcal{F}\times\mathcal{D}} \tr((H_z{E_z^*}A^*{E_w}H_w^*) ({H_w}E_w^*{A}E_{z}H_z))d\mu^2(z,w).
\end{aligned}
\end{equation*}
\end{small}
Here we are able to take the integral over the product space $\mathcal{F}\times\mathcal{D}$ since the integral is finite. 
Similarly, we obtain
\begin{equation*}
\begin{aligned}\tau(AA^*)=\frac{1}{ c_{\pi}\cdot\mu(\mathcal{F})}\int_{\mathcal{F}\times\mathcal{D}} \tr((H_z{E_z^*}A{E_w}H_w^*) ({H_w}E_w^*{A}^*E_{z}H_z))d\mu^2(z,w).
\end{aligned}
\end{equation*}
Consider the diagonal action of $\Gamma$ on $\mathcal{D}\times \mathcal{D}$. 
Both of the two integrations are over a $\Gamma$-fundamental domain of $\mathcal{D}\times \mathcal{D}$. 
As the measure $\mu^2$ is $\Gamma$-invariant and the integrand is $\Gamma$-invariant by Lemma \ref{linvtr}, we replace it with the integration over another fundamental domain $(z,w)\in\mathcal{D}\times\mathcal{F}$, which is equivalent to swapping $z,w$. 
Note the integrand is $\Gamma$-invariant under the same action of $\Gamma$ on $\mathcal{D}\times \mathcal{D}$ by 
Lemma \ref{linvtr}. 
Hence the integration above is invariant if we swap $z,w$. 
This is to say $\tau(A^*A)=\tau(AA^*)$ and $\tau$ is a trace. 

Note $\{E_zH_z^*v_i|z\in \mathcal{D},1\leq i\leq n\}$ spans a dense subspace of $H_{\pi}$. 
If $A\neq 0$, we have $\|AE_{z_0}H_{z_0}^*v_i\|^2>0$ for some $z_0\in\mathcal{D}$ and $i$.
As $AE_{z}H_{z}^*v_i$ is continuous in $z$, we have $\|AE_{z}H_{z}^*v_i\|^2>0$ in a neighbourhood $N_{z_0}$ of $z_0$ whose measure $\mu(N_{z_0})$ is strictly positive. 
From the equality above, we also have
\begin{center}
$\tau(A^*A)=\frac{1}{n\cdot c_{\pi}\cdot\mu(\mathcal{F})}\int_{\mathcal{F}}\sum_{i=1}^{n}\langle {A}E_{z}H_z^*v_i,{A}E_{z}H_z^*v_i\rangle_{H_{\pi}}d\mu(z)$
\end{center}
Hence $\tau(A^*A)>0$ and $\tau$ is faithful. 

For the normality, it suffices to prove $\tau$ is completely additive \cite{J15}. 
Take a family of mutually orthogonal projections $\{p_j|j\in J\}$ in $A_{\pi}$ and let $p=\sum_{j\in J}p_i$. 
We have
\begin{center}
 $\tau(p)=\frac{1}{\mu(\mathcal{F})}\int_{\mathcal{F}}\tr(S(p)(z))d\mu(z)=\frac{1}{\mu(\mathcal{F})}\int_{\mathcal{F}}\sum_{j\in J}\tr(S(p_j)(z))d\mu(z)$,   
\end{center}
which converges since $\tau(p)\leq \tau(1)= 1$. 
Moreover, as $0\leq\tr(S(p_j)(z))\leq \tr(S(p)(z))$, we have $\tr(S(p_j)(z))\in L^1(\mathcal{F},\mu)$. 
By the Fubini Theorem, we obtain
\begin{equation*}
\begin{aligned}
\frac{1}{\mu(\mathcal{F})}\int_{\mathcal{F}}\sum_{j\in J}\tr(S(p_j)(z))d\mu(z)&=\sum_{j\in J}\frac{1}{\mu(\mathcal{F})}\int_{\mathcal{F}}\tr(S(p_j)(z))d\mu(z)\\&=\sum_{j\in J}\tau(p_j).
\end{aligned}
\end{equation*}
Hence $\tau$ is normal. 
\end{proof}

\subsection{Toeplitz operators of matrix-valued functions}
\label{steomat}

In this section, we define the matrix Toeplitz operators associated with $\End(V_{\pi})$-valued functions on the bounded symmetric domain $\mathcal{D}$.  
Then we focus on the $\Gamma$-invariant case and give another formula of the trace of the Toeplitz operator. 

Let $H_{\pi}$ be the holomorphic discrete series (or square-integrable) representations defined in Section 2. 
Recall that $P_{\pi}$ is the orthogonal projection from $L^2(\mathcal{D},V_{\pi})$ onto $H_\pi=L^2_{\text{hol}}(\mathcal{D},V_{\pi})$ and $H_z\in GL(V_{\pi})$ defined in Section \ref{sBeretr}. 

Now we consider a measurable $\End(V_{\pi})$-valued function $f$ on $\mathcal{D}$. 
Please note for the measurable space $(\mathcal{D},\mu)$, we call a function $f\colon \mathcal{D}\to \End(V_{\pi})$ measurable if $f=[f_{i,j}]_{1\leq i,j\leq n}$ and $f_{i,j}\colon X\to \mathbb{C}$ is measurable for all $1\leq i,j\leq n$. 

For any $\phi \in L^2(\mathcal{D},V_{\pi})$, one may wonder that when the multiplication operator $M_f\colon \phi\mapsto f\cdot\phi$ is bounded. 
Indeed, we have
\begin{equation*}
\begin{aligned}
\langle f\cdot\phi,f\cdot\phi\rangle_{L^2(\mathcal{D},V_{\pi})}&=\int_{\mathcal{D}}\langle H_zf(z)\phi(z),H_zf(z)\phi(z)\rangle_{\pi}d\mu(z)\\
&=\int_{\mathcal{D}}\langle H_zf(z)H_z^{-1}H_z\phi(z),H_zf(z)H_z^{-1}H_z\phi(z)\rangle_{\pi}d\mu(z)\\
&\leq \int_{\mathcal{D}}\|H_zf(z)H_z^{-1}\|^2_{\text{op}}\langle H_z\phi(z),H_z\phi(z)\rangle_{\pi}d\mu(z),
\end{aligned}
\end{equation*}
where $\|\cdot\|_{\text{op}}$ is the operator norm on the finite dimensional Hilbert space $V_{\pi}$. 
Hence if $\|H_zf(z)H_z^{-1}\|_{\text{op}}$ is essentially bounded on $\mathcal{D}$, say
\begin{center}
$\|H_zf(z)H_z^{-1}\|_{\text{op}}\leq C$ 
\end{center}
for all $z\in\mathcal{D}$. 
So $\|M_f\|_{L^2(\mathcal{D},V_{\pi})}\leq C$. 

Let $\|A\|_F=\Tr(A^*A)^{1/2}$ be the {\it Frobenius norm} of a square matrix $A$ where $\Tr$ is the trace that is not normalized (or the sum of the diagonal elements). 
We define the following two spaces
\begin{enumerate}
    \item $L^{\infty}_H(\mathcal{D},\End(V_{\pi}))=\{f\colon \mathcal{D}\to \End(V_{\pi})\text{~measurable}|\|H_zf(z)H_z^{-1}\|_{\text{op}}\in L^{\infty}(\mathcal{D})\}$. 
    \item  $L^{\infty}(\mathcal{D},\End(V_{\pi}))=\{f\colon \mathcal{D}\to \End(V_{\pi})\text{~measurable}|\|f\|_F\in L^{\infty}(\mathcal{D})\}$. 
\end{enumerate}

\begin{lemma}
We have $f(z)\in L^{\infty}_H(\mathcal{D},\End(V_{\pi}))$ iff $\|H_zf(z)H_z^{-1}\|_F$ is essentially bounded i.e., $H_zf(z)H_z^{-1}\in L^{\infty}(\mathcal{D},\End(V_{\pi}))$. 
\end{lemma}
\begin{proof}
It follows by the fact that the  operator norm is always bounded by the Frobenius norm (i.e., $\|A\|_{\text{op}}\leq  \|A\|_F$ for square matrix $A$). 
\end{proof}
\begin{definition}
For any $f\in L^{\infty}_H(\mathcal{D},\End(V_{\pi}))$, we define the Toeplitz operator in $B(H_\pi)$ associated to $f$ by
\begin{center}
$T_f=P_{\pi}\cdot M_f \cdot P_{\pi}=P_{\pi}\cdot M_f$.
\end{center}
where $M_f$ is the multiplication operator by $f$ on $H_{\pi}$. 
More precisely, for any $\phi\in H_\pi$, the operator acts on it by
\begin{center}
$(T_f\circ \phi)(z)=P_{\pi}(f\circ \phi)(z)$
\end{center}
where $(f\circ\phi)(z)=f(z)\phi(z)\in V_{\pi}$. 
\end{definition}
\begin{remark}
When $f$ takes values in the center of $\End(V_{\pi})$, it can be identified with a scalar-valued function.
In this case, $T_f$ is just the classical Toeplitz operator associated with $f\in L^{\infty}(\mathcal{D})$.  
For more details on classical Toeplitz operators associated with functions on the open unit disk, we refer to \cite{BS,HKZ,Va}.
\end{remark}

\begin{proposition}\label{pTfppt}
We have $T_f\in B(H_{\pi})$ for each $f\in L^{\infty}_H(\mathcal{D},\End(V_{\pi}))$ with the following properties: 
\begin{enumerate}[label=(\roman*)]
    \item $T\colon f\mapsto T_f$ is linear, i.e., $T_{\alpha f+\beta h}=\alpha T_f+\beta T_h$ for $\alpha,\beta\in \mathbb{C}$ and $f,h\in L^{\infty}_H(\mathcal{D},\End(V_{\pi}))$,
    \item $T^*_{f(z)}=T_{H_z^{-1}(H_z^{*})^{-1}f(z)^*H_z^*H_z}$ where $f^*$ is given pointwise by $f^*(z)=f(z)^*$ acting on $V_{\pi}$. 
    In particular, if $f\in L^{\infty}(\mathcal{D})$, $T_f^*=T_{\overline{f}}$. 
\end{enumerate}
\end{proposition}
\begin{proof}
As $M_f$ is bounded, it is clear that $T_f=P_{\pi}{M_f}P_{\pi}$ is bounded on $H_{\pi}$. 
The linearity is straightforward. 

For the adjoint $T_f^*$, let $\phi,\eta\in H_{\pi}$ and consider the following inner product: 
\begin{equation*}
\begin{aligned}
\langle \phi,T^*_f\eta\rangle_{H_\pi}&=\langle T_f\phi,\eta\rangle_{H_\pi}=\langle M_f\phi,\eta\rangle_{L^2}\\
&=\int_{\mathcal{D}}\langle H_zf(z)\phi(z),H_z\eta(z)\rangle_{\pi}d\mu(z)\\
&=\int_{\mathcal{D}}\langle H_zf(z)H_z^{-1}H_z\phi(z),H_z\eta(z)\rangle_{\pi}d\mu(z)\\
&=\int_{\mathcal{D}}\langle H_z\phi(z),(H_z^{-1})^{*}f(z)^*H_z^*H_z\eta(z)\rangle_{\pi}d\mu(z)\\
&=\int_{\mathcal{D}}\langle H_z\phi(z),H_z(H_z^{-1}(H_z^{-1})^{*}f(z)^*H_z^*H_z)\eta(z)\rangle_{\pi}d\mu(z)\\
&=\langle \phi(z),M_{H_z^{-1}(H_z^{-1})^{*}f(z)^*H_z^*H_z}\eta(z)\rangle_{L^2}\\
&=\langle \phi(z),T_{H_z^{-1}(H_z^{-1})^{*}f(z)^*H_z^*H_z}\eta(z)\rangle_{H_{\pi}}. 
\end{aligned}
\end{equation*}
This implies 
$T^*_{f(z)}=T_{H_z^{-1}(H_z^{-1})^{*}f(z)^*H_z^*H_z}$.    
\end{proof}

Now we consider the left action of $G$ on $\mathcal{D}$ and also on $L^{\infty}(\mathcal{D},\End(V_{\pi}))$ given by $g\cdot f(z)=f(g^{-1}z)$. 
\begin{proposition}
\label{pTfga}
For any $f\in L^{\infty}_H(\mathcal{D},\End(V_\pi))$,
We have 
\begin{center}
$L_{\pi}(g)T_{f(z)}L_{\pi}(g)^*=T_{\pi(J(g^{-1},z))^{-1}(g\cdot f)(z)\pi(J(g^{-1},z))}$.
\end{center}
Hence if $f(g^{-1}\cdot z)=\pi(J(g^{-1},z))f(z)\pi(J(g^{-1},z))^{-1}$ for all $\gamma\in \Gamma,z\in \mathcal{D}$, then $T_f$ commutes with the action of $\Gamma$. 
In particular, for $f\in L^{\infty}(\mathcal{D})$, $T_f$ commutes with the action of $\Gamma$ if $f$ is $\Gamma$-invariant. 
\end{proposition}
\begin{proof}
Let $\phi \in H_\pi$. 
We have
\begin{equation*}
\begin{aligned}
&(L_{\pi}(g)T_{f}L_{\pi}(g)^*\phi)(z)\\
=&L_{\pi}(g)P_{\pi}{M_f}\pi(J(g,z))^{-1}\phi(gz)=P_{\pi}L_{\pi}(g)f(z)\pi(J(g,z))^{-1}\phi(gz)\\
=&P_{\pi}\pi(J(g^{-1},z))^{-1}f(g^{-1}z)\pi(J(g^{-1},z))\pi(J(g,z))^{-1}\pi(J(g,g^{-1}z))^{-1}\phi(z)\\
=&P_{\pi}M_{\pi(J(g^{-1},z))^{-1}f(g^{-1}z)\pi(J(g^{-1},z))}\phi(z)\\
=&T_{\pi(J(g^{-1},z))^{-1}(g\cdot f)(z)\pi(J(g^{-1},z))}\phi(z),  
\end{aligned}
\end{equation*}
where we use $\pi(J(g,z))^{-1}\pi(J(g,g^{-1}z))^{-1}=I$. 
Hence 
\begin{center}
$L_{\pi}(g)T_{f(z)}L_{\pi}(g)^*=T_{\pi(J(g^{-1},z))^{-1}(g\cdot f)(z)\pi(J(g^{-1},z))}$.    
\end{center} 
If $f\in L^{\infty}(\mathcal{D})$, i.e., $f(z)\in \mathbb{C}$, we have $f(z)=g\cdot f(z)=f(g^{-1}z)$.  
\end{proof}

Now we define $R\colon \mathcal{D}\times\mathcal{D}\to \End(V_{\pi})$ by
\begin{center}
$R(w,z)=H_z{E_z^*}E_w{H_w^*}$. 
\end{center}
Note $R(w,z)^*=R(z,w)$ and it is indeed the element $c_{\pi}\cdot R(I)(z,w)$ in $\End(V_\pi)$. 
Moreover, we let $\delta\colon \mathcal{D}\times\mathcal{D}\to \End(V_{\pi})$ given by
\begin{center}
$\delta(z,w)=R(w,z)R(w,z)^*=H_z{E_z^*}E_w{H_w^*}H_w{E_w^*}E_z{H_z^*}$, 
\end{center}
which is a positive operator in $\End(V_{\pi})$. 
\begin{lemma}\label{lstf}
For $f\in L^{\infty}_H(\mathcal{D},\End(V_\pi))$, we have
\begin{center}
$S(T_f)(z)=\frac{1}{c_{\pi}}\int_{\mathcal{D}}R(w,z)(H_wf(w)H_w^{-1})R(w,z)^*d\mu(w)$.
\end{center}
If $f\in L^{\infty}(\mathcal{D})$, 
$S(T_f)(z)=\frac{1}{c_{\pi}}\int_{\mathcal{D}}f(w)\delta(z,w)d\mu(w)$. 
\end{lemma}
\begin{proof}
Take any $u,v\in V_{\pi}$ and consider the following inner product. 
\begin{equation*}
\begin{aligned}
\langle S(T_f)(z)u,v\rangle_{\pi}
=&\frac{1}{c_\pi}\langle T_f{E_z}H_z^*u, E_z{H_z^*}v\rangle_{H_{\pi}}=\frac{1}{c_\pi}\langle M_f{E_z}H_z^*u, E_z{H_z^*}v\rangle_{L^2}\\
=&\frac{1}{c_\pi}\int_{\mathcal{D}}\langle H_w^* H_w f(w)E_w^* {E_z}H_z^*u, E_w^* E_z{H_z^*}v\rangle_{\pi}d\mu(w)\\
=&\frac{1}{c_\pi}\int_{\mathcal{D}}\langle H_z E_z^* E_w H_w^* H_w f(w)E_w^* {E_z}H_z^*u, v\rangle_{\pi}d\mu(w)\\
=&\langle \frac{1}{c_{\pi}}\int_{\mathcal{D}}R(w,z)(H_wf(w)H_w^{-1})R(w,z)^*d\mu(w)u,v\rangle_{\pi}
\end{aligned}
\end{equation*}
If $f(z)\in \mathbb{C}\cdot I\in \End(V_{\pi})$, we obtain $R(w,z)(H_wf(w)H_w^{-1})R(w,z)^*=f(w)\delta(z,w)$. 
\end{proof}

Now we identify the $\Gamma$-invariant function in $L^{\infty}(\mathcal{D})$ with $L^{\infty}(\mathcal{F})$, i.e., $L^{\infty}(\mathcal{F})=L^{\infty}(\mathcal{D})^{\Gamma}$. 
By Proposition \ref{pTfga}, the Toeplitz operator gives a map $T\colon L^{\infty}(\mathcal{F})\to B(H_{\pi})^{\Gamma}=A_{\pi}$ by $f\mapsto T_f$. 

For the $\End(V_\pi)$-valued Toeplitz operators, consider the extension form $L^{\infty}_H(\mathcal{F},\End(V_\pi))$ to $L^{\infty}_H(\mathcal{D},\End(V_\pi))$ given in Proposition \ref{pTfga} by
\begin{center}
$\gamma\cdot f(z)=f(\gamma^{-1}z)=\pi(J(\gamma^{-1},z))f(z)\pi(J(\gamma^{-1},z))^{-1}$, $\forall \gamma\in \Gamma$, $z\in\mathcal{F}$
\end{center}
This establishes a map form $L^{\infty}_H(\mathcal{F},\End(V_\pi))$ to $A_{\pi}$:
\begin{center}
$T\colon L^{\infty}_H(\mathcal{F},\End(V_\pi))\to B(H_\pi)^\Gamma=A_{\pi}$.  
\end{center}
We also denote by $T_f$ the Toeplitz operator given by this extension of $f$.   
\begin{proposition}
\label{ptr2}
Given $A\in A_{\pi}$ and $f\in L^{\infty}_H(\mathcal{F},\End(V_\pi))$, we have
\begin{center}
$\tau(AT_f)=\frac{1}{\mu(\mathcal{F})}\int_{\mathcal{F}}\tr(f(z)Q(A)(z))d\mu(z)$. 
\end{center}
\end{proposition}
\begin{proof}
We let $v_1,\dots,v_n$ be an orthonormal basis of $V_{\pi}$. 
By Lemma \ref{lproj} and Proposition \ref{ptr1}, we obtain:
\begin{small}
\begin{equation*}
\begin{aligned}
\tau(AT_f)=&\frac{1}{\mu(\mathcal{F})}\tr(S(AT_f)(z))d\mu(z)\\
=&\frac{1}{n\cdot c_{\pi} n\cdot \mu(\mathcal{F})}\int_{\mathcal{F}}\sum_{i=1}^{n}\langle H_z E_z^* AT_f E_z H_z^* v_i,v_i\rangle_{\pi}d\mu(z)\\
=&\frac{1}{n c_{\pi}\mu(\mathcal{F})}\sum_{i=1}^{n}\int_{\mathcal{F}}\langle T_f E_z H_z^* v_i,A^*E_zH_z^*v_i\rangle_{H_\pi}d\mu(z)\\
=&\frac{1}{n c_{\pi}\mu(\mathcal{F})}\sum_{i=1}^{n}\int_{\mathcal{F}}\langle M_f E_z H_z^* v_i,A^*E_zH_z^*v_i\rangle_{L^2}d\mu(z)\\
=&\frac{1}{n c_{\pi}\mu(\mathcal{F})}\sum_{i=1}^{n}\int_{\mathcal{F}}\bigl(\int_{\mathcal{D}}\langle H_w f(w)E_w^* E_z H_z^* v_i,H_wE_w^*A^*E_zH_z^*v_i\rangle_{\pi}d\mu(w)\bigr)d\mu(z)\\
=&\frac{1}{ c_{\pi}\mu(\mathcal{F})}\int_{\mathcal{F}}\bigl(\int_{\mathcal{D}} \tr(H_zE_z^*AE_wH_w^*H_w f(w)E_w^* E_z H_z^*)d\mu(w)\bigr)d\mu(z)
\end{aligned}
\end{equation*}
\end{small}
As in the proof of Proposition \ref{ptr1}, we consider the diagonal action of $\Gamma$ on $\mathcal{D}^2$. 
The fundamental domain is $\mathcal{F}\times\mathcal{D}$. 
Following the proof of Lemma \ref{linvtr}, we can show the integrand $\tr(H_zE_z^*AE_wH_w^*H_w f(w)E_w^* E_z H_z^*)$ is $\Gamma$-invariant. 
Since $\mu^2$ is a $\Gamma$-invariant measure, we can also replace it with another fundamental domain $\mathcal{D}\times\mathcal{F}$ by changing $(z,w)$ to $(w,z)$, which leaves the integration invariant. 
Hence the integral above equals to:  
\begin{equation*}
\begin{aligned}
&\frac{1}{ c_{\pi}\mu(\mathcal{F})}\int_{\mathcal{F}}\bigl(\int_{\mathcal{D}} \tr(H_zE_z^*AE_wH_w^*H_w f(w)E_w^* E_z H_z^*)d\mu(z)\bigr)d\mu(w)\\
=&\frac{1}{n c_{\pi}\mu(\mathcal{F})}\sum_{i=1}^{n}\int_{\mathcal{F}}\bigl( \int_{\mathcal{D}}\langle H_zE_z^*AE_wH_w^*v_i,H_zE_z^*E_wf(w)^*H_w^*v_i\rangle_{\pi}d\mu(z)\bigr)d\mu(w)\\
=&\frac{1}{n c_{\pi}\mu(\mathcal{F})}\sum_{i=1}^{n}\int_{\mathcal{F}}\langle AE_wH_w^*v_i,E_wf(w)^*H_w^*v_i\rangle_{H_\pi}d\mu(w)\\
=&\frac{1}{ c_{\pi}\mu(\mathcal{F})}\int_{\mathcal{F}} \tr(H_wf(w)E_w^*AE_wH_w^*)d\mu(w)\\
=&\frac{1}{ \mu(\mathcal{F})}\int_{\mathcal{F}} \tr(H_wf(w)H_w^{-1}S(A)(w))d\mu(w)\\
=&\frac{1}{ \mu(\mathcal{F})}\int_{\mathcal{F}} \tr(K_A(w)H_w^*H_wf(w))d\mu(w)
\end{aligned}
\end{equation*}
\end{proof}

If $f\in L^{\infty}(\mathcal{D})$, this formula of the trace can be simplified as follows. 
\begin{corollary}
For $f\in L^{\infty}(\mathcal{F})$, we have 
\begin{center}
$\tau(AT_f)=\frac{1}{\mu(\mathcal{F})}\int_{\mathcal{F}}f(z)\tr(Q(A)(z))d\mu(z)$. 
\end{center}
\end{corollary}
\begin{proof}
It follows the fact $\tr(f(z)Q(A)(z))=f(z)\tr(Q(A)(z))$ if $f(z)$ is a scalar. 
\end{proof}

\section{The Commutant of the Group von Neumann Algebras}

\subsection{The $L^2$-space of matrix-valued functions}

We keep the notations as in the previous sections. 
Let $\mu$ be the measure on $\mathcal{F}$ obtained from $(\mathcal{D},\mu)$.  
For a measurable function $f\colon \mathcal{F}\to \End(V_{\pi})$, we denote by $f_H$ the following function
\begin{center}
$f_H(z)=H_z\cdot f(z)\cdot H_z^{-1}$, $z\in \mathcal{F}$. 
\end{center}
Consider the following vector space of $\End(V)$-valued functions on $\mathcal{F}$: 
\begin{footnotesize}
\begin{center}
$L^2_H(\mathcal{F},\End(V))=\{f\colon \mathcal{F}\to \End(V)\text{ measurable}|\int_{\mathcal{F}}||f_H(z)||_F^2 d\mu(z)<\infty\}$. 
\end{center}
\end{footnotesize}
Here $\Tr$ is the trace on $\End(V)$ which is not normalized and $||\cdot||_F$ is the Frobenius norm. 
We also denote a sesquilinear form defined on $L^2_H(\mathcal{F},\End(V),\mu)$ by  
\begin{center}
$\langle f,h\rangle=\langle f,h\rangle_{L^2_H}=\int_{\mathcal{F}}\Tr(f_H(z)h_H(z)^*)d\mu(z)$
\end{center}
where $f,h\in L^2_H(\mathcal{F},\End(V_\pi),\mu)$.  
\begin{lemma}\label{lsesLH}
The sesquilinear form $\langle \cdot,\cdot\rangle$ is an inner product on $L^2_H(X,\End(V))$.  
\end{lemma}
\begin{proof}
We can check $\langle \alpha f_1+\beta f_2,h\rangle= \alpha\langle f_1 ,h\rangle+\beta\langle f_2,h\rangle$ for $\alpha,\beta\in \mathbb{C}$.  
As 
\begin{equation*}
\begin{aligned}
\Tr(f_H(z)h_H(z)^*)&=\Tr(H_zf(z)H_z^{-1}(H_z^{-1})^*h(z)^*H_z^*)\\
&=\overline{\Tr(H_zh(z)H_z^{-1}(H_z^{-1})^*f(z)^*H_z^*)}=\overline{\Tr(h_H(z)f_H(z)^*)}.
\end{aligned}
\end{equation*}
We have $\langle f,h\rangle=\overline{\langle h,f\rangle}$. 

Now we assume $\langle f,f\rangle=0$. 
Then $\|H_zf(z)H_z^{-1}\|_F=0$ almost everywhere on $\mathcal{F}$. 
Hence $f\stackrel{\text{a.e.}}{=}0$. 
\end{proof}

Therefore we obtain a Hilbert space $L^2_H(\mathcal{F},\End(V),\mu)$, or simply denote it by $L^2_H(\mathcal{F},\End(V))$. 

Now we consider the following space
\begin{center}
$L^{\infty}_H(\mathcal{F},\End(V_{\pi}))=\{f\colon \mathcal{F}\to \End(V_{\pi})\text{~measurable}|\|f_H(z)\|_{F}\in L^{\infty}(\mathcal{F})\}$.  
\end{center}
One can show this is a complex algebra by the fact $(fh)_
H(z)=f_H(z)h_H(z)$ and the Frobenius norm $\|\cdot\|_F$ is sub-multiplicative, i.e., $\|AB\|_F\leq \|A\|_F\cdot \|B\|_F$. 
We denote $\|(\|f_H(z)\|_F)\|_{\infty}$, which is the essential norm of $\|f_H(z)\|_F\in L^{\infty}(\mathcal{F})$, by $\|f\|_{H,\infty}$.  
Note that $\|f_H(z)\|_{\text{op}}\leq \|f\|_{H,\infty}$ for all $z\in\mathcal{F}$. 

Furthermore, we also define another Hilbert space
\begin{center}
$L^{2}_H(\mathcal{F},V_{\pi})=\{\phi\colon \mathcal{F}\to V_{\pi}\text{~measurable}|\int_{\mathcal{F}}\langle H_z\phi(z),H_z\phi(z)\rangle_{\pi}d\mu(z)<\infty\}$,   
\end{center}
where the inner product is given as $\langle\cdot,\cdot\rangle_1=\int_{\mathcal{F}}\langle H_z\phi(z),H_z\psi(z)\rangle_{\pi}d\mu(z)$ for $\phi,\psi\in L^{2}_H(\mathcal{F},V_{\pi})$. 

Define an action $\sigma$ of $f\in L^{\infty}_H(\mathcal{F},\End(V_{\pi}))$ on $L^{2}_H(\mathcal{F},V_{\pi})$ as
\begin{center}
$\sigma(f(z))\circ\phi(z)=f(z)\cdot \phi(z)$, 
\end{center}
where $f\in L^{\infty}_H(\mathcal{F},\End(V_{\pi}))$ and $\phi \in L^{2}_H(\mathcal{F},V_{\pi})$. 

\begin{proposition}\label{pLHact}
The action $\sigma$ defined above gives a well-defined faithful $C^*$-representation of the algebra $L^{\infty}_H(\mathcal{F},\End(V_{\pi}))$ such that
\begin{enumerate}
    \item The adjoint of $\sigma(f(z))$ is $\sigma(f(z))^*=\sigma(H_z^{-1}(H_z^*)^{-1}f(z)^*H_z^*H_z)$,
    \item $\sigma(f(z))$ is a positive operator iff $H_zf(z)H_z^{-1}=g(z)^*g(z)$ for some $g\in L^{\infty}(\mathcal{F},\End(V_{\pi}))$.
    
\end{enumerate}
\end{proposition}
\begin{proof}
Assume $\|f\|_{H,\infty}=C$. 
Then we have 
\begin{equation*}
\begin{aligned}
\langle H_zf(z)\phi(z),H_zf(z)\phi(z)\rangle_{\pi}&=\langle f_H(z)H_z\phi(z),f_H(z)H_z\phi(z)\rangle_{\pi}\\
&\leq \|f_H(z)\|_{\text{op}}\cdot\langle H_z\phi(z),H_z\phi(z)\rangle_{\pi}\\
&\leq C\cdot \langle H_z\phi(z),H_z\phi(z)\rangle_{\pi}. 
\end{aligned}
\end{equation*}
Hence $\|\sigma(f)\|_{L^{2}_H(\mathcal{F},V_{\pi})}\leq C$ which is well-defined. 

It is straightforward to check  $\sigma(f)=0$ if and only if $f=0$. 
Moreover, we have
\begin{equation*}
\begin{aligned}
\langle \sigma(f)\phi,\psi\rangle_H&=\int_{\mathcal{F}}\langle H_zf(z) \phi(z),H_z\psi(z)\rangle_{\pi}d\mu(z)\\
&=\int_{\mathcal{F}}\langle H_zf(z)H_z^{-1}H_z \phi(z),H_z\psi(z)\rangle_{\pi}d\mu(z)\\
&=\int_{\mathcal{F}}\langle H_z \phi(z),H_zH_z^{-1}(H_z^*)^{-1}f(z)^*H_z^*H_z\psi(z)\rangle_{\pi}d\mu(z).
\end{aligned}
\end{equation*}
This proves $\sigma(f(z))^*=\sigma(H_z^{-1}(H_z^*)^{-1}f(z)^*H_z^*H_z)$. 
By the second line above, $\sigma(f(z))$ is positive iff $H_zf(z)H_z^{-1}=g(z)^*g(z)$ for some $g\in L^{\infty}(\mathcal{F},\End(V_{\pi}))$. 
\end{proof}

As $\sigma$ is faithful, we denote also by $L^{\infty}_H(\mathcal{F},\End(V_{\pi}))$ its image under $\sigma$ and equipped it with the $C^*$-structure as above. 

Now we regard $L^{\infty}_H(\mathcal{F},\End(V_{\pi}))$ as the $\Gamma$-invariant functions in $L^{\infty}_H(\mathcal{D},\End(V_{\pi}))$ (see Proposition \ref{pTfga}).  
For a given $f\in L^{\infty}_H(\mathcal{F},\End(V_{\pi}))$, we denote by $\tilde{f}$ its $\Gamma$-invariant lifting to $L^{\infty}(\mathcal{D},\End(V_{\pi}))$ as follows.  
For $w\in\mathcal{D}$, there is a unique $\gamma\in \Gamma$ such that $w=\gamma z$. 
Let $\gamma_{\pi}(z)=\pi(J(\gamma,z))=\pi(\kappa(\gamma\exp z))$, the function $\tilde{f}(w)$ with $w\in \mathcal{D}$ is given as following:  
\begin{center}
$\tilde{f}(w)=\tilde{f}(\gamma z)=\gamma_{\pi}(z)f(z)\gamma_{\pi}(z)^{-1}$, for all $\gamma\in \Gamma,z\in \mathcal{F}$.  
\end{center}
The Toeplitz operator $T_{\tilde{f}}$ will be denote simply by $T_f$ as Section \ref{steomat}. 

\begin{lemma}\label{lextFD}
For $z\in\mathcal{F}$, we have 
\begin{center}
$H_{\gamma z}\tilde{f}(\gamma z)H_{\gamma z}^{-1}=\pi(k(\gamma,z))^*H_zf(z)H_z^{-1}\pi(k(\gamma,z))$    
\end{center}
with some $k(\gamma,z)\in K$. 
Hence the $\End(V_{\pi})$-valued function $\tilde{f}$ is in $L^{\infty}_H(\mathcal{D},\End(V_{\pi}))$ and gives a well-defined $\Gamma$-intertwining Toeplitz operator $T_{\tilde{f}}$. 
\end{lemma}
\begin{proof}
Recall $H_z=\pi(h_z^{-1})=\pi(\kappa(g_z)^{-1})$ (see Section \ref{sBeretr}). 
In the proof of Proposition \ref{pS}.(v), we know there is $k(\gamma,z)\in K$ such that $k(\gamma,z)=h_z^{-1}\kappa(\gamma\exp z)^{-1}h_{\gamma z}$. 
Then $\pi(k(\gamma,z))=H_z\pi(\kappa(\gamma\exp z)^{-1})H_{\gamma z}^{-1}$. 
We obtain
\begin{equation*}
\begin{aligned}
H_{\gamma z}\tilde{f}(\gamma z)H_{\gamma z}^{-1}&=\pi(k(\gamma,z))^{-1}H_z\pi(\gamma\exp z)^{-1})\tilde{f}(\gamma z)\pi(\gamma\exp z))H_z^{-1}\pi(k(\gamma,z))\\
&=\pi(k(\gamma,z))^{-1}H_zf(z)H_z^{-1}\pi(k(\gamma,z)). 
\end{aligned}
\end{equation*}
Note $k(\gamma,z)\in K$ and $\pi$ is a unitary representation of $K$. 
We have $\|H_{\gamma z}\tilde{f}(\gamma z)H_{\gamma z}^{-1}\|_{\text{op}}=\|H_zf(z)H_z^{-1}\|_{\text{op}}$. 

The $\Gamma$-intertwining property follows from the definition of $\tilde{f}$ and Proposition \ref{pTfga}. 
\end{proof}

Now we define a map $B$ on $L^{\infty}_H(\mathcal{F},\End(V_\pi))$ by
\begin{center}
$Bf(z)=\frac{1}{c_\pi}E_z^*T_{f}E_zH_z^*H_z$, $z\in \mathcal{D}$, 
\end{center}
for $f\in L^{\infty}_H(\mathcal{F},\End(V_\pi))$. 
It is related to the Berezin symbols by $Bf(z)=H_z^{-1}S(T_{f})(z)H_z$. 
We denote the $\pi(k(\gamma,z))^*$ by $k_{\pi}(\gamma,z)$.

\begin{lemma}\label{ldefB}
For $f\in L^{\infty}_H(\mathcal{F},\End(V_\pi))$, we have $Bf(z)\in L^{\infty}_H(\mathcal{F},\End(V_\pi))$ which can be given as
\begin{equation*}
\begin{aligned}
Bf(z)&=H_z^{-1}\int_{\mathcal{F}}\frac{1}{c_{\pi}}\sum_{\gamma\in \Gamma}R({\gamma w},z)(H_{\gamma w}\gamma_{\pi}(w)f(w)\gamma_{\pi}(w)^{-1}H_{\gamma w}^{-1})R({\gamma w},z)^*d\mu(w)H_z\\
&=H_z^{-1}\int_{\mathcal{F}}\frac{1}{c_{\pi}}\sum_{\gamma\in \Gamma}R({\gamma w},z)k_{\pi}(\gamma,z)(H_{w}f(w)H_{ w}^{-1}))k_{\pi}(\gamma,z)^*R({\gamma w},z)^*d\mu(w)H_z
\end{aligned}
\end{equation*}
Furthermore, if we take $f=I_{V_\pi}$, $BI_{V_\pi}(z)=I_{V_\pi}$.  
\end{lemma}
\begin{proof}
As $Bf(z)=H_z^{-1}S(T_f)(z)H_z$, it follows then by Proposition \ref{pTfga} and Proposition \ref{lextFD}. 
Then the case $f=I_{V_{\pi}}$ is straightforward by Proposition \ref{pS}. 
\end{proof}

\begin{proposition}\label{pextB}
The map $B$ defined on $L^{\infty}_H(\mathcal{F},\End(V_\pi))$ can be extended to a bounded operator on $L^2_H(\mathcal{F},\End(V_\pi))$. 
\end{proposition}
\begin{proof}
As $\mu(\mathcal{F})<\infty$, $L^{\infty}_H(\mathcal{F},\End(V_{\pi}))$ is a dense subspace of $L^{2}_H(\mathcal{F},\End(V_{\pi}))$. 
It suffices to show
\begin{center}
$\|Bf\|^2_{L^{2}_H(\mathcal{F},\End(V_{\pi}))}\leq C\cdot \|f\|^2_{L^{2}_H(\mathcal{F},\End(V_{\pi}))}$
\end{center}
for any $f\in L^{\infty}_H(\mathcal{F},\End(V_{\pi}))$. 

Take any $z\in \mathcal{F}$ and consider the following map
\begin{equation*}
\begin{aligned}
\phi_z\colon ~~&L^{\infty}_H(\mathcal{F},\End(V_\pi)) &\to&~~~ \End(V_\pi)&\\
&~~~~~~~~~~~~~f&\mapsto&~~~\phi_zf=H_z\cdot Bf(z)\cdot H_z^{-1}&
\end{aligned}
\end{equation*}
We first show $\phi_z$ is a unital positive map. 
By 2 of Proposition \ref{pLHact}, we assume $\sigma(f)$ (or simply $f$) is positive, i.e., $H_zf(z)h_z^{-1}=g(z)^*g(z)$ for some $g\in L^{\infty}(\mathcal{F},\End(V_{\pi}))$. 
Hence $\phi_z$ is positive as the image 
\begin{center}
$\phi_z(f)=\int_{\mathcal{F}}\frac{1}{c_{\pi}}\sum_{\gamma\in \Gamma}R({\gamma w},z)k_{\pi}(\gamma,z)^*(g(w)^*g(w))R({\gamma w},z)^*d\mu(w)$     
\end{center}
is positive in $\End(V_{\pi})$.
Furthermore, $\phi_z$ is unital by Lemma \ref{ldefB}.

Witout loss of generality, we may assume $f$ is normal (or self-adjoint).
Then, by Kadison's inequality \cite{Ka52}, we have
\begin{equation*}
\begin{aligned}
&\phi_z(f)\phi_z(f^*)\leq \phi_z(ff^*)\\
=&\int_{\mathcal{F}}\frac{1}{c_{\pi}}\sum_{\gamma\in \Gamma}R({\gamma w},z)k_{\pi}(\gamma,z)^*(H_{w}f(w)H_w^{-1}(H_w^*)^{-1}f(w)^*H_w^*H_wH_{w}^{-1})R({\gamma w},z)^*d\mu(w)\\
=&\int_{\mathcal{F}}\frac{1}{c_{\pi}}\sum_{\gamma\in \Gamma}R({\gamma w},z)k_{\pi}(\gamma,z)^*(H_{w}f(w)H_w^{-1}(H_w^*)^{-1}f(w)^*H_w^*)R({\gamma w},z)^*d\mu(w). 
\end{aligned}
\end{equation*}

Consider the $L^2_H$-norm of $Bf$. 
Note that $R(w,z)^*=R(z,w)$ and $\|R(w,z)\|_F^2$ is $\Gamma$-invariant by Lemma \ref{linvtr}. 
Hence we have
\begin{align*}
&\|Bf\|^2_{L^{2}_H(\mathcal{F},\End(V_{\pi}))}\\
=&\int_{\mathcal{F}}\Tr(H_zBf(z)H_z^{-1}(H_z^{-1})^*Bf(z)^*H_z^*)d\mu(z)\\
=&\int_{\mathcal{F}}\Tr(\phi_z(f)\phi_z(f^*))d\mu(z)\leq \int_{\mathcal{F}}\Tr(\phi_z(ff^*))d\mu(z)\\
=&\int_{\mathcal{F}}\Tr(\int_{\mathcal{F}}\frac{1}{c_{\pi}}\sum_{\gamma\in \Gamma}R({\gamma w},z)k_{\pi}(\gamma,z)^*(H_{w}f(w)H_w^{-1}(H_w^*)^{-1}f(w)^*H_w^*)k_{\pi}(\gamma,z)R({\gamma w},z)^*d\mu(w))d\mu(z)\\
=&\frac{1}{c_{\pi}}\int_{\mathcal{F}}\int_{\mathcal{F}}\sum_{\gamma\in \Gamma}\Tr(R({\gamma w},z)k_{\pi}(\gamma,z)^*(H_{w}f(w)H_w^{-1}(H_w^*)^{-1}f(w)^*H_w^*)k_{\pi}(\gamma,z)R({\gamma w},z)^*)d\mu(w)d\mu(z)\\
=&\frac{1}{c_{\pi}}\int_{\mathcal{F}}\int_{\mathcal{F}}\sum_{\gamma\in \Gamma}\|R({\gamma w},z)k_{\pi}(\gamma,z)^*H_{w}f(w)H_w^{-1}\|_F^2 d\mu(w)d\mu(z)\\
\leq &\frac{1}{c_{\pi}}\int_{\mathcal{F}}\int_{\mathcal{F}}\sum_{\gamma\in \Gamma}\|R({\gamma w},z)\|_F^2 \|H_{w}f(w)H_w^{-1}\|_F^2 d\mu(w)d\mu(z)\\
= &\frac{1}{c_{\pi}}\int_{\mathcal{F}}\bigl(\|H_{w}f(w)H_w^{-1}\|_F^2\cdot \int_{\mathcal{F}}\sum_{\gamma\in \Gamma}\|R({\gamma w},z)\|_F^2  d\mu(z)\bigr)d\mu(w)\\
= &\frac{1}{c_{\pi}}\int_{\mathcal{F}}\bigl(\|H_{w}f(w)H_w^{-1}\|_F^2\cdot \int_{\mathcal{F}}\sum_{\gamma\in \Gamma}\|R({ w},\gamma^{-1}z)\|_F^2  d\mu(z)\bigr)d\mu(w)\\
= &\frac{1}{c_{\pi}}\int_{\mathcal{F}}\bigl(\|H_{w}f(w)H_w^{-1}\|_F^2\cdot \int_{\mathcal{F}}\sum_{\gamma\in \Gamma}\Tr(R({ w},\gamma^{-1}z)R({ w},\gamma^{-1}z)^*)  d\mu(z)\bigr)d\mu(w)\\
= &\frac{1}{c_{\pi}}\int_{\mathcal{F}}\bigl(\|H_{w}f(w)H_w^{-1}\|_F^2\cdot \int_{\mathcal{F}}\sum_{\gamma\in \Gamma}\Tr(R({ \gamma^{-1}z},w)^*R({ \gamma^{-1}z},w)^*)  d\mu(z)\bigr)d\mu(w)\\
= &\frac{1}{c_{\pi}}\int_{\mathcal{F}}\bigl(\|H_{w}f(w)H_w^{-1}\|_F^2\cdot \Tr(\int_{\mathcal{F}}\sum_{\gamma\in \Gamma}R({ \gamma^{-1}w},z)^*R({ \gamma^{-1}w},z)^*  d\mu(z))\bigr)d\mu(w)\\
=&\int_{\mathcal{F}} n\cdot\|H_{w}f(w)H_w^{-1}\|_F^2\mu(w)\\
=& n\cdot \|f\|^2_{L^{2}_H(\mathcal{F},\End(V_{\pi}))},
\end{align*}
where we also apply  $BI_{V_{\pi}}(z)=I_{V_{\pi}}$ (see Lemma \ref{ldefB}). 
Hence $B$ is bounded on the $L^2$-space.  
\end{proof}
\begin{corollary}\label{cBinj}
The operator $B$ is injective on $L^{\infty}_H(\mathcal{F},\End(V_\pi))$. 
\end{corollary}
\begin{proof}
Take $f\in L^{\infty}_H(\mathcal{F},\End(V_\pi))$, we know $H_zf(z)H_z^{-1}\in L^{\infty}(\mathcal{F},\End(V_\pi)$. 
Rewrite $H_zf(z)H_z^{-1}=g(z)+i\cdot h(z)$ for some $f,g\in L^{\infty}(\mathcal{F},\End(V_\pi))$ such that $g(z)^*=g(z),h(z)^*=h(z)$ for all $z\in\mathcal{F}$. 

We assume $Bf=0$. 
From the proof of \ref{pextB}, we
know 
\begin{center}
$\|Bf\|^2_{L^{2}_H(\mathcal{F},\End(V_{\pi}))}
=\int_{\mathcal{F}}\Tr(\phi_z(f)\phi_z(f^*))d\mu(z)=0$.  
\end{center} 
So $\phi_z(f)=0$ for all $z$, which is to say
\begin{align*}
\Tr(\phi_z(f))&=\int_{\mathcal{F}}\frac{1}{c_{\pi}}\sum_{\gamma\in \Gamma}R({\gamma w},z)k_{\pi}(\gamma,z)(g(w)+i\cdot h(w))k_{\pi}(\gamma,z)^*R({\gamma w},z)^*d\mu(w)\\
&=\Tr\bigl(\int_{\mathcal{F}}\frac{1}{c_{\pi}}\sum_{\gamma\in \Gamma}R({\gamma w},z)k_{\pi}(\gamma,z)g(w)k_{\pi}(\gamma,z)^*R({\gamma w},z)^*d\mu(w)\bigr)\\
&~~~+i\cdot\Tr\bigl(\int_{\mathcal{F}}\frac{1}{c_{\pi}}\sum_{\gamma\in \Gamma}R({\gamma w},z)k_{\pi}(\gamma,z) h(w)k_{\pi}^*(\gamma,z)R({\gamma w},z)^*d\mu(w)\bigr). 
\end{align*}
Hence it imlplies $g(w)=h(w)=0$ for all $w\in \mathcal{F}$ and $f=0$. 
\end{proof}



\begin{proposition}
\label{pex2}
The map $T\colon L^{\infty}_H(\mathcal{F},\End(V_\pi))\to A_{\pi}$ given by $f\mapsto T_{f}$
can be extended to a bounded linear operator $L^2_H(\mathcal{F},\End(V_\pi))\to L^2(A_{\pi},\tau)$. 
\end{proposition}
\begin{proof}
Note $L^{\infty}_H(\mathcal{F},\End(V_\pi))$ is dense in $L^{2}_H(\mathcal{F},\End(V_\pi))$ as $\mu(\mathcal{F})$ is finite. 
Take $f\in L^{\infty}_H(\mathcal{F},\End(V_\pi))$ and consider the trace $\tau(T_f^* T_f)$. 
\begin{equation*}
\begin{aligned}
\|T_f\|_{L^2(A_\pi,\tau)}^2&=\tau(T_f^{*}T_f)
=\frac{1}{\mu(\mathcal{F})}\int_{\mathcal{F}}\tr(f(z)Q(T_f^*)(z))d\mu(z)\\
&=\frac{1}{n\cdot\mu(\mathcal{F})}\langle f,Bf\rangle_{L^2_H(\mathcal{F},\End(V_\pi))}\leq \frac{\|B\|_{L^2_H(\mathcal{F},\End(V_\pi))}}{n\cdot\mu(\mathcal{F})}\|f\|^2_{L^2_H(\mathcal{F},\End(V_\pi))}.
\end{aligned}
\end{equation*}
Hence $T$ can be extended to a bounded operator on $L^2_H(\mathcal{F},\End(V_{\pi}))$: 
\begin{center}
$T\colon L^2_H(\mathcal{F},\End(V_\pi))\to L^2(A_{\pi},\tau)$    
\end{center}
with $\|T\|\leq \frac{\|B\|_{L^2}}{n\cdot\mu(\mathcal{F})}$, which is bounded by $1$ as shown in Proposition \ref{pextB}. 
\end{proof}

\begin{corollary}\label{cTT}
$T^*(A)=\frac{1}{n\cdot\mu( \mathcal{F})}Q(A)$
\end{corollary}
\begin{proof}
Let us consider $\langle T^*(A),f \rangle_{L^2_H(\mathcal{F},\End(V_\pi))}$ for an arbitrary $f\in L^{\infty}_H(\mathcal{F},\End(V_\pi))$. 
By Proposition \ref{ptr2} and Proposition \ref{pTfppt} .(ii), we have
\begin{equation*}
\begin{aligned}
\langle T^*(A),f \rangle_{L^2_H(\mathcal{F},\End(V_\pi))}&=\langle A,T_f\rangle_{L^2(A_\pi,\tau)}=\tau(AT_f^*)=\tau(AT_{H_z^{-1}(H_z^{-1})^*f(z)^*H_z^*H_z})\\
&=\frac{1}{\mu(\mathcal{F})}\int_{\mathcal{F}}\tr(K_A(z)H_z^*H_zH_z^{-1}(H_z^{-1})^*f(z)^*H_z^*H_z)d\mu(z)\\
&=\frac{1}{\mu(\mathcal{F})}\int_{\mathcal{F}}\tr(H_z(K_A(z)H_z^*H_z)H_z^{-1}\cdot (H_z^{-1})^*f(z)^*H_z^*)d\mu(z)\\
&=\frac{1}{\mu(\mathcal{F})}\int_{\mathcal{F}}\tr(H_zQ(A)(z)H_z^{-1}\cdot (H_z^{-1})^*f(z)^*H_z^*)d\mu(z)\\
&=\frac{1}{n\cdot\mu( \mathcal{F})}\int_{\mathcal{F}}\Tr(H_zQ(A)(z)H_z^{-1}\cdot (H_z^{-1})^*f(z)^*H_z^*)d\mu(z)\\
&=\langle \frac{1}{n\cdot\mu( \mathcal{F})}Q(A)(z),f(z)\rangle_{L^2_H(\mathcal{F},\End(V_\pi))} 
\end{aligned}    
\end{equation*}
As the $L^{\infty}$-space is dense in $L^2_H(\mathcal{F},\End(V_\pi))$, this implies $T^*(A)=\frac{1}{n\cdot\mu( \mathcal{F})}Q(A)$. 
\end{proof}

\begin{proposition}\label{psdense}
The range of $T$ is dense in $L^2(A_{\pi},\tau)$. 
\end{proposition}
\begin{proof}
It suffices show $T^*$ is injective on $L^2(A_{\pi},\tau)$. 
Let $\nu$ be the measure $\frac{c_{\pi}}{\mu(\mathcal{F})}(\mu\times \mu)$ on $\mathcal{F}\times \mathcal{D}$.  
Consider the following Hilbert space
\begin{center}
$K=L^2(\mathcal{F}\times \mathcal{D},\End(V_\pi),\nu)=\{f\colon \mathcal{F}\times \mathcal{D}\to\End(V_\pi)|\langle f,f\rangle_K<\infty\}$.
\end{center}
Here the inner product is given by 
\begin{center}
$\langle f,h\rangle_K=\int_{\mathcal{F}\times \mathcal{D}}\tr(f(z,w)h(z,w)^*)d\nu(z,w)$.     
\end{center}
We can check this gives an inner product which makes $K$ a Hilbert space. 

For any $A,B\in A_{\pi}$, we have
\begin{equation*}
\begin{aligned}
\tau(AB^*)&=\frac{1}{c_{\pi}\mu(\mathcal{F})}\sum_{1\leq i\leq n}\int_{\mathcal{F}}\langle H_z{E_z^*}B^*A{E_z}H_z^* v_i,v_i\rangle_{\pi}d\mu(z)\\
&=\frac{1}{c_{\pi}\mu(\mathcal{F})}\sum_{1\leq i\leq n}\int_{\mathcal{F}}\langle A{E_z}H_z^* v_i,B{E_z}H_z^*v_i\rangle_{H_{\pi}}d\mu(z)\\
&=\frac{1}{c_{\pi}\mu(\mathcal{F})}\sum_{1\leq i\leq n}\int_{\mathcal{F}}\left(\int_{\mathcal{D}}\langle H_w{E_w^*}A{E_z}H_z^* v_i,H_w{E_w^*}B{E_z}H_z^*v_i\rangle_{\pi}d\mu(w)\right)d\mu(z)\\
&=\int_{\mathcal{F}\times\mathcal{D}}\tr(R(A)(z,w)(R(B)(z,w))^{*}d\nu(z,w).
\end{aligned}
\end{equation*}
Hence $\langle A,B\rangle_{\tau}=\langle R(A),R(B)\rangle_K$ and $R$ is an isometry from $L^2(A_\pi,\tau)$ to $K$, i.e. $R^*R=\text{id}$.  

Note we have $T^*(A)=\frac{1}{n\cdot \mu(\mathcal{F})}Q(A)$ by Corollary \ref{cTT}. 
Hence the map $T^*R^*$ on $R(A_{\pi})$ is exactly the map given by
\begin{center}
$T^*R^*\colon R(A)(z,w)\mapsto \frac{1}{n\cdot\mu(\mathcal{F})}Q(A)(z)$, $A\in A_{\pi}$. 
\end{center}
Note each element in $R(A_{\pi})$ can be written as $\frac{1}{c_{\pi}}H_wK_A(w,z)H_z^*$ for some $A\in A_{\pi}$. 
So this map is exactly the map given by
\begin{center}
$\frac{1}{c_{\pi}}H_wK_A(w,z)H_z^*\mapsto \frac{1}{n\cdot c_{\pi}\cdot\mu(\mathcal{F})}K_A(z,z)H_z^* H_z$
\end{center}
It can be extended to a well-defined bounded map on $R(L^2(A_\pi,\tau))$, the range of $R$. 
Note by Proposition \ref{pS}, the map $R(A)\mapsto Q(A)$ is injective. 

For any $A\in A_{\pi}$, by Proposition \ref{pK}, we know $c_{\pi}H_w^{-1}R(A)(z,w)(H_z^*)^{-1}=K_A(w,z)$ is sesqui-holomorphic (in $w,\bar{z}$.) 
Hence that the range $R(L^2(A_{\pi}))$ are also in the following space
\begin{equation*}
\begin{aligned}
&L^2_{\text{ses-holo}}(\mathcal{F}\times \mathcal{D},\nu)=\{h(z,w)\in K|H_w^{-1}h(z,w)(H_z^*)^{-1} \text{is sesqui-holomorphic}\}.
\end{aligned}
\end{equation*}
Consider the following diagram. 
\begin{equation*}
\begin{CD}
f@>T>> T_f @>{R}>> R(T_f)(z,w)\\
L^{\infty}_H(\mathcal{F},\End(V_{\pi}))
 @>T>>
 A_{\pi}=B(H_n)^{\Gamma}@>{R}>>L^2_{\text{ses-holo}}(\mathcal{F}\times \mathcal{D},\nu)\\
@VV{i}V
 @VV{i}V
 @VV{i}V\\
  L^{2}_H(\mathcal{F},\End(V_{\pi}))
 @>T>>
L^2(A_{\pi},\tau)@>{R}>>L^2_{\text{ses-holo}}(\mathcal{F}\times \mathcal{D},\nu)
\end{CD}
\end{equation*}
Note the map $T^*R^*$ on $R(A_{\pi})$ is given above.
We obtain the explicit formula for $T^*R^*$ on $R(L^2(A_{\pi}))$ given by 
\begin{center}
$T^*R^*\colon h(z,w)\mapsto q(z)= \frac{1}{n\cdot \mu(\mathcal{F})}H_z^{-1} h(z,z)H_z $, $\forall h(z,w)\in R(L^2(A_{\pi}))$.  
\end{center}
Note $H_w^{-1}h(z,w)(H_z^*)^{-1}$ is sesqui-holomorphic. 
By Theorem \ref{tsesholo}, we know the map
\begin{center}
$H_w^{-1}h(z,w)(H_z^*)^{-1}\mapsto H_z^{-1}h(z,z)(H_z^*)^{-1}$    
\end{center}
is injective. 
Hence the map $h(z,w)\mapsto q(z)H_z^{-1}(H_z^*)^{-1}$ is injective and so is
$T^*R^*\colon h(z,w)\mapsto q(z)$. 
This implies the injectivity of $T^*$ on $L^2(A_{\pi},\tau)$. 
\end{proof}

\subsection{The commutant and its subalgebras}

We will show the operators $T_f$ with $f\in L^{\infty}_H(\mathcal{F},\End(V_\pi))$ generate the commutant $A_{\pi}$ of the group von Neumann algebra $L_{\pi}(\Gamma)''$.  
We recall a well-known fact: 
\begin{lemma}
\label{ltopo}
Let $M\subset B(H)$ be a von Neumann algebra with a positive, faithful, normal, normalized trace $\tr$. 
Then the topology induced by $\|x\|_{2}=\tr(xx^*)^{1/2}$ coincides with the strong operator topology on any bounded subset of $M$. 
\end{lemma}
\begin{proof}
It suffices to consider the unit ball $M_1=\{x\in M|\|x\|_{B(H)}\leq 1\}$ in  $M$. 
Let us consider the GNS construction for $\tr$ and we get a normal faithful representation
\begin{center}
$\pi_{\tr}\colon M\to B(L^2(M,\tr))$.
\end{center}
Note $\pi_{\tr}$ is injective,  strong operator topology to strong operator topology continuous and $\|\pi_{\tr}(x)\|\leq \|x\|$. 

Take a sequence $\{x_i\}_{i\geq 1}$ in $M_1$ and suppose we $\|x_i\|_2\rightarrow 0$. 
Let $\Omega=\hat{1}$ be the cyclic vector in $L^2(M,\tr)$. 
We take an arbitrary $y\in M$ and $\hat{y}=y\Omega\in L^2(M,\tr)$. 
\begin{equation*}
\begin{aligned}
\|\pi_{\tr}(x_i)\hat{y}\|_2^2&=\langle\pi_{\tr}(x_i)y\Omega,\pi_{\tr}(x_i)y\Omega\rangle_{L^2(M,\tr)}=\tr(x_iyy^*x_i^*)\\
&\leq \|y\|^2\tr(x_ix_i^*)\rightarrow 0. 
\end{aligned}    
\end{equation*}
Note $M$ is $\|\cdot\|_2$-dense in $L^2(M,\tr)$ by the GNS construction. 
For any $v\in L^2(M,\tr)$ and positive integer $N$, there exists a $y\in M$ such that $\|v-\hat{y}\|<\frac{1}{N}$. 
We have $\|\pi_{\tr}(x_i)v\|_2\leq \|x_i\|\|v-\hat{y}\|_2+\|\pi_{\tr}(x_i)\hat{y}\|_2\leq \frac{1}{N}+\|\pi_{\tr}(x_i)\hat{y}\|_2$.  
Hence $\pi_{\tr}(x_i)v\to 0$ and we can apply $\pi_{\tr}^{-1}$ (which is also strong operator continuous, see \cite{KR2} 7.1.16) so that $x_i\to 0$ in the strong operator topology on $B(H)$. 

Conversely, if $x_i\rightarrow 0$ in the strong operator topology on $B(H)$, we have $\pi_{\tr}(x_i)$ also converges to $0$ in the strong operator topology on $L^2(M)$. 
Then 
\begin{center}
$\|x_i\|_2^2=\tr(x_ix_i^*)=\langle \pi_{\tr}(x_i)\Omega, \pi_{\tr}(x_i)\Omega\rangle_{L^2(M)}\rightarrow 0$.      
\end{center}
\end{proof}

As shown above, $T_f\in A_{\pi}$ for $f\in L^{\infty}(\mathcal{F},\End(V_{\pi}))$. 
A natural question is how large is the subalgebra of $A_\pi$ generated by these operators.  
\begin{proposition}\label{peqdense}
Let $M\subset B(H)$ be a von Neumann algebra with a positive, faithful, normal trace $\tr$ and $A\subset M$ be a $*$-subalgebra of $M$. 
Then $A$ is $L^2$-dense in $L^2(M,\tr)$ if and only if it is weak operator dense in $M$, i.e.
\begin{center}
$\overline{A}^{||\cdot||_2}=L^2(M)$ if and only if $\overline{A}^{\operatorname{w.o}}=M$.    
\end{center}
\end{proposition}
\begin{proof}
As the norm topology is finer than the weak operator topology, we assume $A$ is norm closed, i.e., $A$ is a $C^*$-algebra. 
Since $A$ is convex, $\overline{A}^{\operatorname{w.o}}=\overline{A}^{\operatorname{s.o}}$. 

Take any self-adjoint $x\in M$ with $\|x\|\leq 1$. 
There exists a net $\{a_n\}_{n\geq 1}$ in $A$ such that $\|a_n-x\|_2\rightarrow 0$. 
Also we have $\|a_n^*-x^*\|_2\rightarrow 0$ as $\|a\|_2=\tr(aa^*)=\tr(a^*a)=\|a^*\|_2$ for all $a\in M$.  
So $\|\frac{a_n+a_n^*}{2}-x\|_2\rightarrow 0$ and hence we can further assume $\{a_n\}_{n\geq 1}$ are self-adjoint, i.e. $a_n\in A_{\text{s.a}}$. 

Consider $f(t)=\frac{2t}{t^2+1}$ which is a bijection on $[-1,1]$. 
Let $g=f^{-1}$ and $y=g(x)$ then $y\in M_{\text{s.a}}$. 
By the argument above, there are $\{b_n\}_{n\geq 1}$ in $A_{\text{s.a}}$ such that $\|b_n-y\|_2\rightarrow 0$. 

We want to show $\|f(b_n)-x\|_2\rightarrow 0$. 
Note that $f(b_n)\in A$ and $\|f(b_n)\|\leq 1$, hence by continuous functional calculus, we have
\begin{center}
$f(b_n)-f(y)=\frac{2(b_n(1+y^2)-y(1+b_n^2))}{(1+b_n^2)(1+y^2)}=\frac{2(b_n-y)}{(1+b_n^2)(1+y^2)}+\frac{2b_{n}y(y-b_n)}{(1+b_n^2)(1+y^2)}$.
\end{center}
Note that $\|ab\|_2\leq \|a\|\|b\|_2$ and $\|ab\|_2\leq \|b\|\|a\|_2$. 
Moreover, $\|(1+b_n^2)^{-1}\|\leq 1$ and $\|(1+y^2)^{-1}\|\leq 1$.  
Hence
\begin{equation*}
\begin{aligned}
\|f(b_n)-f(y)\|_2&\leq 2\|(1+y^2)^{-1}\|\cdot\|b_n-y\|_{2}+2\|f(b_n)\|\|y(1+y^2)^{-1}\|\cdot\|b_n-y\|_{2}\\
&\leq 4\|b_n-y\|_{2}\rightarrow 0. 
\end{aligned}
\end{equation*}
Hence $\|f(b_n)-x\|_2=\|f(b_n)-f(y)\|_2\rightarrow 0$.  

Note $\|f(b_n)\|\leq 1$. 
That is to say the closure of unit ball of $A_{\text{s.a}}$ (inside $M$) in $\|\cdot\|_2$ is just the unit ball of $M_{\text{s.a}}$. 
By Lemma \ref{ltopo}, we obtain $\overline{(A_{\text{s.a}})_1}^{\operatorname{s.o}}=(M_{\text{s.a}})_1$ and hence 
$\overline{A}^{\operatorname{w.o}}=\overline{A}^{\operatorname{s.o}}=M$. 

For the converse, it suffices to prove $\overline{A_1}^{||\cdot||_2}=\overline{M_1}^{||\cdot||_2}$ or equivalently for any $x\in M_1$, there exists a sequence $\{x_k\}_{k\geq 1}$ in $A_1$ such that $x_k\xrightarrow{\|\cdot\|_{2}} x$. 
This is guaranteed by the assumption  $\overline{A}^{\operatorname{s.o}}=M$, Lemma \ref{ltopo} and also the Kaplansky density thoerem \cite{J15}.  
\end{proof}

Finally we can determine the von Neumann algebra generated by these $T_f$'s.  
\begin{theorem}\label{tgenvna}
The von Neumann algebra generated by the Toeplitz operators associated with the functions in $L^{\infty}_H(\mathcal{F},\End(V_\pi))$ is dense in the commutant $A_{\pi}$ in the strong operator topology, i.e.,
\begin{center}
$\overline{\langle T_f|f\in L^{\infty}_H(\mathcal{F},\End(V_\pi))\rangle}^{\text{s.o.}}=A_{\pi}$. \end{center}
\end{theorem}
\begin{proof}
As $L^{\infty}_H(\mathcal{F},\End(V_\pi))$ is dense in $L^{2}_H(\mathcal{F},\End(V_\pi))$, by Proposition \ref{psdense}, we know $\{T_f|f\in L^{\infty}_H(\mathcal{F},\End(V_\pi))\}$ is a dense subspace of $L^2(A_\pi,\tau)$. 
Then, by Proposition \ref{peqdense}, these $T_f$'s generated $A_{\pi}$ in the strong operator (hence also in the weak operator topology). 
\end{proof}
\begin{corollary}
$\overline{\langle \{T_f|f\in L^{\infty}(\mathcal{F})\}\otimes\End(V_{\pi})\rangle}^{\text{s.o.}}=A_{\pi}$. 
\end{corollary}\label{cTfgen}
\begin{proof}
Take $f\in L^{\infty}_H(\mathcal{F},\End(V_\pi))$ and assume $H_zf(z)H_z^{-1}=g(z)=[g_{i,j}(z)]_{1\leq i,j\leq n}$ with each $g_{i,j}\in L^{\infty}(\mathcal{F})$. 
Then we have
\begin{center}
$T_f=T_{H_z^{-1}g(z)H_z}=T_{\sum_{i,j}H_z^{-1}g_{i,j}(z)e_{i,j}H_z}=\sum_{i,j}T_{g_{i,j}(z)H_z^{-1}e_{i,j}H_z}$. 
\end{center}
Now we define a map between two complex vector spaces: 
\begin{center}
$\Phi\colon\{T_f|f\in L^{\infty}_H(\mathcal{F},\End(V_\pi))\}\to  \{T_f|f\in L^{\infty}(\mathcal{F})\}\otimes \Mat_n(\mathbb{C})$, 
\end{center}
which is given by
\begin{center}
$\Phi\colon T_f=\sum_{i,j}T_{g_{i,j}(z)H_z^{-1}e_{i,j}H_z}\mapsto [T_{g_{i,j}}]_{1\leq i,j\leq n}$. 
\end{center}
It is straightforward to check $\Phi$ is linear and surjective. 

For the injectivity, we suppose there are two $g_{i,j},g_{i,j}'\in L^{\infty}(\mathcal{F})$ such that 
$T_{g_{i,j}(z)H_z^{-1}e_{i,j}H_z}=T_{g_{i,j}'(z)H_z^{-1}e_{i,j}H_z}$ as  Toeplitz operators. 
Then we have $S(T_{g_{i,j}(z)H_z^{-1}e_{i,j}H_z})=S(T_{g_{i,j}'(z)H_z^{-1}e_{i,j}H_z})$. 
Hence $B(g_{i,j}(z)H_z^{-1}e_{i,j}H_z)=B(g_{i,j}'(z)H_z^{-1}e_{i,j}H_z)$. 
Then, by Corollary \ref{cBinj}, we know $g_{i,j}=g_{i,j}'$ as they are scalar functions and hence $\Phi$ is injective. 
\end{proof}

\begin{remark}
For $f\in L^{\infty}(\mathcal{F})$, these $T_f$'s will certainly generate a von Neumann subalgebra of $A_{\pi}$. 
For $n\neq 1$, it is still unknown that how large this subalgebra is.  
This is related to the Toeplitz $C^*$-algebras with continuous symbols on the bounded symmetric domains, see \cite{Upm84}. 

Indeed $\{p_i=M_{H_z^{-1}e_{i,i}H_z}\}_{1\leq i\leq n}$ gives a family of orthogonal projections in $B(L^2(\mathcal{D},V_{\pi}))$ satisfying $\sum_{1\leq i\leq n}p_i=1$. 
One can show
\begin{center}
$\langle M_f|f\in L^{\infty}_H(\mathcal{D},\End(V_{\pi}))\rangle \cong \langle M_f|f\in L^{\infty}(\mathcal{D})\rangle \otimes \End(V_{\pi})$ 
\end{center}
as an isomorphism of von Neumann algebras acting on $L^2(\mathcal{D},V_{\pi}), p_iL^2(\mathcal{D},V_{\pi}),V_{\pi}$ respectively ($e_{i,j}\in \End(V_{\pi})$ acts as $M_{H_z^{-1}e_{i,j}H_z}$). 
Note these $p_i$'s commute with the action of $G$. 
We can also consider the  $\Gamma$-invariant case: 
\begin{center}
$\langle M_f|f\in L^{\infty}_H(\mathcal{F},\End(V_{\pi}))\rangle \cong \langle M_f|f\in L^{\infty}(\mathcal{F})\rangle \otimes \End(V_{\pi})$. 
\end{center}
We may also identify $A_\pi$ with the strong operator topology closure of the algebra generated by $P_{\pi}\cdot\langle M_f|f\in L^{\infty}_H(\mathcal{F},\End(V_{\pi}))\rangle\cdot P_{\pi}$ in $B(H_{\pi})$ by Theorem \ref{tgenvna}.  
 
\end{remark}

\section{An Example on Fuchsian Subgroups of $SL(2,\mathbb{R})$}
\label{ssl2}

The holomorphic discrete series of $G=SL(2,\mathbb{R})$ are indexed by integers that are greater or equal to $2$, say $\{L_m,H_m\}_{m\geq 2}$. 
Given a Fuchsian subgroup $\Gamma$ of $G$, let $f$ be a cusp form for $\Gamma$ of weight $p$. 
We are able to associate a bounded linear operator $T_f\in B(H_m,H_{m+p})$, which intertwines the action of $\Gamma$. 

Let $A_m=\{A\in B(H_m)|AL_m(\gamma)=L_m(\gamma)A,\forall \gamma\in \Gamma\}$, 
which is the commutant of the von Neumann algebra $L_m(\Gamma)''$. 
This section is devoted to the proof of the following result: 
\begin{theorem}\label{tsl2}
Let $\Gamma\subset SL(2,\mathbb{R})$ be a Fuchsian subgroup and $T_f$ be the Toeplitz operator associated with a cusp form $f$ of $\Gamma$. 
Then
\begin{center}
$\overline{\{\text{span}_{f,g} (T_g)^*T_f\}}^{\text{w.o.}}=A_m$    
\end{center}
as $f,g$ run through all cusp forms of $\Gamma$ of same weights. 
\end{theorem}
This generalizes F. Radulescu's result on $SL(2,\mathbb{Z})\subset SL(2,\mathbb{R})$ \cite{Ra94}. 
In the case of $\Gamma=SL(2,\mathbb{Z})$, most of the results are known, see \cite{Ra14,Ra98}. 

Note for $G=SL(2,\mathbb{R})$, $K=SO(2)$ is a maximal subgroup. 
Hence by Theorem \ref{tbddo}, the symmetric domain $\mathcal{D}=G/K$ is just the open unit disk.
For the convenience to discuss automorphic forms, we identify it with the Poincar\'{e} upper-half plane
\begin{center}
$\mathbb{H}=\{z=x+iy\in \mathbb{C}|y> 0\}$. 
\end{center}
with the invariant measure $d\mu=y^{-2}dxdy$. 
Moreover, since $K=SO(2)$ is abelian, all its irreducible representation are one-dimensional and can be characterized as $(\pi_m,V_m)$ such that $\pi_m(g)=g^m\in S^1$ for an integer $m$. 

Note $K_{\mathbb{C}}=\mathbb{C}^{\times}$ is the complexified group of $K=SO(2)$. 
By definition, the canonical automorphy factor $J\colon SL(2,\mathbb{R})\times\mathbb{H}\to \mathbb{C}^{\times}$ is given by
\begin{center}
$J(g,z)=cz+d$,  $g=\bigl(\begin{smallmatrix} a &b \\ c & d \\ \end{smallmatrix}\bigr) \in SL(2,\mathbb{R})$, $z\in \mathbb{H}$. 
\end{center}

Following Section 2.2, we can describe the holomorphic discrete series representations of $SL(2,\mathbb{R})$. 
One can show that $\pi_m(\kappa(z,z)^{-1})=y^m$ where $z=x+iy$. 
So we move the term $y^m$ to the measure and denote $y^md\mu=y^{m-2}dxdy$ by $d\mu_m$. 

Let $L^2(\mathbb{H},\mu_m)$ be the square-integrable functions on $\mathbb{H}$ with respect to the measure $\mu_m=y^{m-2}dxdy$ (note $\mu_0=\mu$). 
Let $H_m$ be the subspace of all holomorphic functions in $L^2(\mathbb{H},\mu_m)$, i.e.,
\begin{center}
$H_m=L^2_{\text{hol}}(\mathbb{H},\mu_m)$.
\end{center}
As in Section \ref{sholds}, for a given $g=\bigl(\begin{smallmatrix} a &b \\ c & d \\ \end{smallmatrix}\bigr)^{-1} \in SL_2(\mathbb{R})$ and $f\in H_m$, the action on $H_m$ is given by
\begin{center}
$(L_m(g)f)(z)=f(g^{-1}z)(cz+d)^{-m}$
\end{center}
where $g^{-1}z=\frac{az+b}{cz+d}$. 

\begin{proposition}[\cite{GHJ}]
For any integer $m\geq 2$,  $(L_m,H_m)$ is an irreducible unitary representation of $SL(2,\mathbb{R})$. 
Moreover, it is square-integrable with formal dimensions $d_m=\frac{m-1}{4\pi}$.
\end{proposition}


\subsection{Berezin transform and the trace}

Let $(L_m,H_m)$ be the holomorphic discrete series of $SL(2,\mathbb{R})$ associated with the one dimensional representation $(\pi_m,V_m)$ of $K=SO(2)$ (and also of $K_{\mathbb{C}}=\mathbb{C}^{\times}$). 
Note all the matrix-valued ($\End(V_\pi)$-valued) functions  defined in Section 3.1 reduce to the  scalar-valued functions since $V_m=\mathbb{C}$ for all $m$.   
In this section, we use the  following simplified notation for Berezin symbols. 
\begin{enumerate}[label=(\roman*)]
\item $K(z,w)=E_z^*E_w$ and  $K(z,z)=E_z^*E_z$.
\item $\widehat{A}(z)=S(A)(z)$ for $A\in B(H_m)$ 
\item $\widehat{A}(z,w)=R(A)(z,w)$ for $A\in B(H_m)$ (defined in the proof of Proposition \ref{psdense}).
\end{enumerate}
Let $v_1\in V_m$ be a unit vector, one can further show $\widehat{A}(z,w)=\frac{\langle AE_w(v_1),E_z(v_1)\rangle_{H_m}}{\langle E_w(v_1),E_z(v_1)\rangle_{H_m}}$. 

\begin{corollary}
Given $A,B\in \mathbb{B}(H_m)$, then
\begin{enumerate}[label=(\roman*)]
\item $\widehat{A}(z,w)$ is sesqui-holomorphic (i.e., holomorphic in $z$ and anti-holomorphic in $w$),
\item the map $A\mapsto \widehat{A}(z,w)$ or $\widehat{A}(z)$ is one-to-one,
\item $\sup_{z\in \mathbb{H}}|\widehat{(A)}(z)|\leq ||A||$,
\item $\widehat{A^{*}}(z,w)=\overline{\widehat{A}(w,z)}$,
\item $\widehat{AB}(z,w)=\int_\mathbb{H}\frac{K(z,\eta)K(\eta,w)}{K(z,w)}
    \widehat{A}(z,\eta)\widehat{B}(\eta,w)d\mu_m(\eta)$.
\end{enumerate}
\end{corollary}
\begin{proof}
It follows from Proposition \ref{pK} and Proposition \ref{pS}. 
\end{proof}

Now we define the commutant by
\begin{center}
$A_m=L_{m}(\Gamma)'\cap B(H_{\pi})=\{A\in B(H_m)|AL_m(\gamma)=L_m(\gamma)A,~\forall \gamma \in \Gamma\}$,   
\end{center}
which will be shown to be a tracial von Neumann algebra. 
It is a $\text{II}_1$ factor if $L_{m}(\Gamma)''$ is a $\text{II}_1$ factor and the coupling constant of $\dim_{L_{m}(\Gamma)''}H_m$ is finite, which holds when $\Gamma$ is an ICC lattice.

\begin{corollary}
\label{pinv}
If $A,B\in \mathbb{B}(H_m)$ and $g\in SL_2(\mathbb{R})$, the Berezin transform of $L_m(g)^{-1}AL_m(g)$ is $\widehat{A}(gz,gw)$. 
$A\in \mathcal{A}_m$ if and only if $\widehat{A}$ is $\Gamma$-invariant.
\end{corollary}
\begin{proof}
It follows from Proposition \ref{pS} (iv).
\end{proof}

\begin{corollary}
\label{ttr1}
For $m\geq 2$, 
the following linear functional defines a faithful normal tracial state on $A_m$: 
\begin{center}
$\tau(A)=\frac{1}{\mu(\mathcal{F})}\int_\mathcal{F}\widehat{A}(z)d\mu(z)$, $A\in A_{m}$,
\end{center}
If $\Gamma$ is an ICC group, $\tau$ is the unique tracial state if $A_m$ is a type $\text{II}_1$ factor. 

\end{corollary}
\begin{proof}
It follows from Proposition \ref{ptr1}. 
\end{proof}

\begin{remark}
Let $Z(\Gamma)$ be the center of $\Gamma$. 
The representation $(L_m,H_m)$ is indeed a projective unitary representation of $\Gamma/Z$, which may also give a factor. 

For example, in the case $PSL(2,\mathbb{Z})=SL(2,\mathbb{Z})/\{\pm I\}$, as $H^2(PSL(2,\mathbb{Z}),S^1)=0$, the representation is ordinary and gives the factor $L_m(PSL(2,\mathbb{Z}))''$ since each Fuchsian subgroup of $PSL(2,\mathbb{R})$ is ICC \cite{A81}. 

One may also consider the lattices in $SL(2,\mathbb{R})$. 
Indeed, they are all essentially ICC: there are only finitely many conjugacy classes that are finite. 
Following \cite{Ra14}, in this case the $A_m$'s are also factors of type $\text{II}_1$. 
\end{remark}

\subsection{The action of cusp forms}

Let $\Gamma$ be a Fuchsian subgroup of the first kind.
Recall that a cusp form of weight $p$ of $\Gamma$ is a holomorphic function $f\colon \mathbb{H}\to \mathbb{C}$ satisfying
\begin{enumerate}[label=(\roman*)]
\item $f(z)=(cz+d)^{-p}f(\frac{az+b}{cz+d})$, $z\in\mathbb{H}$, $\bigl(\begin{smallmatrix} a &b \\ c & d \\ \end{smallmatrix}\bigr) \in
    \Gamma$,
\item $f$ vanishes at each cusp of $\Gamma$. 
\end{enumerate}
One can refer \cite{Mi,Sh} for the precise descriptions.   
Let $S_p(\Gamma)$ be the vector space generated by all cusp forms of weight $p$ of $\Gamma$, which is finite dimensional. 
It is well-known that for any $f\in S_k(\Gamma)$, there is a constant $B_f\geq 0$ such that
$|f(x+iy)|\leq B_f\cdot y^{-p/2}$ \cite{Mi}.  

Let $\mathcal{A}^{0}(\Gamma,\pi_p)$ be the space of cusp forms defined on $SL(2,\mathbb{R})$, see Section \ref{sauto1}. 
Indeed, let $\Delta$ be the Casimir element of $\mathfrak{gl}_2^{\mathbb{C}}$ and $Z$ be the $2$-by-$2$ identity matrix in $\mathfrak{gl}_2$, 
there is an isomorphism \begin{equation*}
\begin{aligned}
S_p(\Gamma)&\to \mathcal{A}^{0}(\Gamma,\langle \Delta-\frac{p^2-1}{4},Z\rangle,\pi_p)\\
f&\mapsto \phi_f(g)=J(g,i)^{-p}f(gi)
\end{aligned}
\end{equation*}
which also illustrates the correspondence between two types of cusp forms. 
We refer to \cite{Bp} for more details about the relation between the automorphic forms on $SL(2,\mathbb{R})$ and the classical automorphic forms defined on the upper-half plane $\mathbb{H}$. 

As representations of $SO(2)$, we have
\begin{center}
$V_p\otimes V_m\cong V_{p+m}$.
\end{center}
Now, given an arbitrary $f\in S_p(\Gamma)$ and any $m\geq 2$, let $T_f=P_{m+p}M_fP_m\in B(H_m,H_{m+p})$ be the Toeplitz operator associated with $f$. 
The following two results are the special cases of the ones in Section \ref{ssintwop}.  
For the reader's convenience, we also give separate proofs that emphasize more scalar-valued cusp forms instead of the vector-valued ones. 
\begin{proposition}
The Toeplitz operator $T_f$ satisfies the following conditions.
\begin{enumerate}[label=(\roman*)]
\item $T_f\in B(H_m,H_{m+p})$,
\item $T_f$ intertwines the action of $\Gamma$, i.e
\begin{center}
$T_f\pi_m(g)=\pi_{m+p}(g)T_f$, $\forall g\in \Gamma$.
\end{center}
\item $(T_g)^*=P_{m}M_{\overline{g}\cdot y^p}P_{m+p}\in B(H_{m+p},H_m)$, which also intertwines  the action of $\Gamma$.
\end{enumerate}
\end{proposition}
\begin{proof}
(i) Let $\phi\in H_m$ and $\psi\in H_{m+p}$. 
Then
\begin{equation*}
\begin{aligned}
\|T_f\phi\|^2_{m+p}&=\int_{\mathbb{H}}|f(z)\phi(z)|^2y^{m+p-2}dxdy\\
&\leq \int_{\mathbb{H}}B^2|\phi(z)|^2y^{m-2}dxdy,
\end{aligned}
\end{equation*}
where we used $|f(z)|\leq By^{-p/2}$. So $T_f \in B(H_m,H_{m+p})$.

(ii) Take $g=\big(\begin{smallmatrix} a &b \\ c & d \\ \end{smallmatrix}\bigr)^{-1} \in \Gamma$.
We have
\begin{equation*}
\begin{aligned}
L_{m+p}(g)T_f\phi(z)&=f\left(\frac{az+b}{cz+d}\right)\phi\left(\frac{az+b}{cz+d}\right)(cz+d)^{-(m+p)}\\
&=f(z)\phi\left(\frac{az+b}{cz+d}\right)(cz+d)^{-m}\\
&=T_f L_m(g)\phi(z).
\end{aligned}
\end{equation*}

(iii) Consider the inner product on $H_{m+p}$. We have
\begin{equation*}
\begin{aligned}
\langle T^{m}_g\phi,\psi\rangle_{m+p}&=\int_{\mathbb{H}}g(z)\phi(z)\overline{\psi(z)}y^{m+p-2}dxdy\\
&=\int_{\mathbb{H}}\phi(z)(\overline{P_m\overline{g}y^p\psi(z)})y^mdxdy,
\end{aligned}
\end{equation*}
which is $\langle\phi,(T_g)^*\psi\rangle_{m}$. Hence ${T_g}^*=P_{m}M_{\overline{g}\cdot y^p}P_{m+p}$.
\end{proof}

\begin{corollary}
Given $f,g\in S_p(\Gamma)$, we have
\begin{center}
$(T_g)^*T_f=P_m M_{f\overline{g}y^p}P_m=T_{f\overline{g}y^p} \in A_m$,
\end{center}
where $A_m=L_m(\Gamma)'$ is the factor.
\end{corollary}
\begin{proof}
Recall $(f(z),g(z))=\langle {\kappa(z,z)}^{-p}f(z),g(z) \rangle_{V_m}=f(z)\overline{g(z)}y^p$.
Note both $f,g$ are holomorphic, then it follows from Corollary \ref{cTeo}.
\end{proof}


\subsection{Two existence results}

This part is devoted to the proof of Theorem \ref{tsl2}. 
Before this, we need two theorems for the existence of some meromorphic functions and holomorphic functions on a compact Riemann surface.  
We refer to \cite{FK} for the theory of Riemann surfaces. 

Let $\Gamma \subset SL(2,\mathbb{R})$ be an arbitrary Fuchsian group of the first kind and $P_{\Gamma}$ be the set of all cusps of $\Gamma$.  
Let $\mathbb{H}^*=\mathbb{H}\cup P_{\Gamma}$. 
Denote $\mathcal{F}=\Gamma\backslash\mathbb{H}$ by the fundamental domain. 
Let $\mathcal{F}^*=\Gamma\backslash\mathbb{H}^*$ be extend fundamental domain. 
It is well-known that $\mathcal{F}^*$ is a compact Hausdorff space and also a compact Riemann surface \cite{Mi}. 
We denote a Riemann surface by $\mathcal{M}$ and the field of meromorphic functions on $\mathcal{M}$ by $A(\mathcal{M})$. 

\begin{theorem}\label{texist1}
If $\mathcal{M}$ is a compact Riemann surface and $P_1,\dots,P_n\in \mathcal{M}$ are distinct points and $z_1,\dots,z_n\in\mathbb{C}$, there exists $\phi\in A(\mathcal{M})$ such that $\phi(P_i)=z_i$ for all $1\leq i\leq n$. 
\end{theorem}
\begin{proof}
Take integers $i,j$ such that $1\leq i\neq j \leq n$. 
Let us consider the divisor $D=k P_i-P_j$ where $k=k_{i,j}\in \mathbb{Z}$. 
Apply the Riemann-Roch Theorem for the divisor $D$, we get
\begin{center}
$l(D)=\deg(D)-g+1+l(\text{div}(\omega)-D)=m-g+l(\text{div}(\omega)-D)$,
\end{center}
where $l(D)=\dim_{\mathbb{C}}L(D)$ with $L(D)=\{f\in A(\mathcal{M})|f=0~{\text{or}}~ \text{div}(f)+D\geq 0\}$ and $\text{div}(\omega)$ is a canonical divisor. 
Take $k$ sufficiently large, there must be a desired $k=k_{i,j}$ such that 
$\deg(\text{div}(\omega)-kP_i+P_i)< 0$ and hence $l(\text{div}(\omega)-kP_i+P_i)=0$. 
Then, as
$l(kP_i-P_j)>l((k-1)P_i-P_j)$, 
there must be some $\psi_{i,j}\in L(kP_i-P_j)-L((k-1)P_i-P_j)$. 
So we get a meromorphic function $\psi_{i,j}$ with
\begin{center}
$v_{P_i}(\psi_{i,j})=-k_{i,j}<0$ and $v_{P_j}(\psi_{i,j})\geq 1$.  
\end{center}
Let $\phi_{i,j}=\frac{\psi_{i,j}}{\psi_{i,j}+1}$ then $\phi_{i,j}(P_i)=1$ and $\phi_{i,j}(P_j)=0$. 
Now we define $\phi_i=\prod_{1\leq j\leq n,j\neq i}\phi_{i,j}$ which satisfies
\begin{center}
$\phi_i(P_i)=1,\phi_i(P_j)=0$ for $j\neq i$. 
\end{center}
Then the function $\phi=\sum_{1\leq i\leq n}z_i\phi_i$ is the one we want. 

\end{proof}

Now we take the compact Riemann surface to be $\mathcal{F}^*=\Gamma\backslash\mathbb{H}^*$. 

\begin{theorem}\label{texist2}
Let $P_1,\dots,P_n\in \mathcal{F}^*$ be distinct points, then there exists a holomorphic automorphic form $f$ such that $f(P_i)\neq 0$ for all $1\leq i\leq n$. 
Moreover, if $\{P_i\}_{1\leq i\leq n}$ are not cusps, we can further require $f$ to be a cusp form.
\end{theorem}
\begin{proof}
We first focus on a single point $P_1$.  
Suppose all holomorphic automorphic forms have $P$ as zeros. 
We take a holomorphic $f_1$ such that $v_{P}(f_1)=m\geq 1$ is minimal. 

Consider the divisor $kP$. 
By the Riemann-Roch Theorem, we have
\begin{center}
$l(kP)=\deg(kP)-g+1+l(\text{div}(\omega)-kP)=k-g+1+l(\text{div}(\omega)-kP)$. 
\end{center}
For sufficiently large $k$, we have $\deg(\text{div}(\omega)-kP)<0$ and $l(kP)=k-g+1$. 
Then there exists $\phi_1\in A(\mathcal{F}^*)$
with a single pole of order $k$. 
Then $g_1=f_1^k \phi_1^{m}$ is holomorphic and does not vanish at $P_1$. 

Now suppose we have such holomorphic automorphic forms $\{g_i\}_{1\leq i\leq n}$ such that $g_{i}(P_i)\neq 0$ and the weight of $g_i$ is $k_i$ for $1\leq i\leq n$. 
Now let $N$ be a common multiple of all these $k_i$'s. 
Then a linear combinations $f=\sum_{1\leq i\leq n}\lambda_i g_i^{N/k_i}$ (with some suitable $\lambda_i$'s) will give us a desired automorphic form of weight $N$.  

If $\{P_i\}_{1\leq i\leq n}$ are not cusps, we can further assume from the beginning that all $g_i$'s are cusp forms. Then we get a desired $f$ in the same way. 
\end{proof}

\subsection{$\text{II}_1$ factors from cusp forms on $SL(2,\mathbb{R})$}

Now we let $A_k(\Gamma)$ be the space of automorphic forms of weight $k$ so 
$S_k(\Gamma)$ is a subspace of $A_k(\Gamma)$ spanned by the cusp forms. 
We know the Petersson inner-product on $S_k(\Gamma)$ given by
\begin{center}
$\langle f,g\rangle=\frac{1}{\mu(\mathcal{F})}\int_{\mathcal{F}} f(z)\overline{g(z)}y^k d\mu(z)$, 
\end{center}
is Hermitian.  
We denote the term in the integral by $(f,g)_k=f(z)\overline{g(z)}y^k$ with a emphasis on the weight $k$. 

Now let 
\begin{center}
$\mathcal{F}_1=\mathcal{F}^*/(\Gamma\backslash P_{\Gamma})=\Gamma\backslash(\mathbb{H}^*/P_{\Gamma})=\mathcal{F}\cup\{\text{pt}\}$.    
\end{center}
by identifying all cusps in the fundamental domain with one point. 
Hence $\mathcal{F}_1$ is compact. 

\begin{proposition}
$\{(f,g)_k|k\in \mathbb{N},f,g\in S_k(\Gamma)\}$ are well-defined functions on $\mathcal{F}_1$ and separate points of $\mathcal{F}_1$. 
\end{proposition}
\begin{proof}
As all $(f,g)_k$'s vanish on all cusps, it is well-defined on the quotient space of $\mathcal{F}^*$ by identifying all cusps. 

Now we take a pair of distinct points $P,Q\in \mathcal{F}_1$. 
By Theorem \ref{texist1}, there is a meromorphic function $\phi\in A_0(\Gamma)$ such that $\phi(P),\phi(Q)$ are distinct. 

Case 1: $P,Q\in \mathcal{F}$. 

We take $f(z),g(z)\in S_k(\Gamma) $ with all $f(P),f(Q),g(P),g(Q)$ are nonzero. 
The existence follows from Theorem \ref{texist2}.   

For sufficiently large weight $k$, we may assume the multiplication by $f$ eliminates all the poles of $\phi$.
So we can further assume $f(z)$ satisfies $f(z)\phi(z)\in S_k(\Gamma)$. 
We assume $(f,g)_k$ cannot separate $P,Q$, i.e., 
\begin{center}
$(f,g)_k(P)=f(P)\overline{g(P)}y_P^k=f(Q)\overline{g(Q)}y_Q^k=(f,g)_k(Q)$.   
\end{center}
Then we have
\begin{center}
$(f\phi,g)_k(P)=f(P)\phi(P)\overline{g(P)}y_P^k\neq f(Q)\phi(Q)\overline{g(Q)}y_Q^k=(f\phi,g)_k(Q)$.
\end{center}

Case 2: $P\in \mathcal{F}$, $Q=\{\text{pt}\}$. 

As $Q$ stands for cusps, it suffices to show some $(f,g)_k(P)\neq 0$. But this follows from Theorem \ref{texist2}. 
\end{proof}

As $\mathcal{F}_1$ is compact, we apply Stone-Weierstrass to get the following corollary. 

\begin{corollary}\label{ccf}
The functions of the form $(f,g)_k$ generate the function space of continuous functions on $\mathcal{F}_1$ that vanish at the the point $\text{pt}$, or equivalently,
$(f,g)_k$'s generate the space of continuous functions on $\mathcal{F}^*$ that vanish on all cusps, i.e., 
\begin{center}
$\overline{\{(f,g)_k|k\in \mathbb{N},f,g\in S_k(\Gamma)\}}^{||\cdot||_{\infty}}=\{\psi\in C(\mathcal{F}^*)~|~\psi|_{\emph{cusps}}=0\}$.  
\end{center}
\end{corollary}

As there are only finitely many cusps in $\mathcal{F}$ and $\mu(\mathcal{F})<\infty$, we obtain: 

\begin{proposition}
For $m\geq 2$, we  have
\begin{center}
$\overline{\{\emph{span}_{f,g} (T_g)^*T_f\}}^{\emph{w.o.}}=A_m$,       
\end{center}  
where $f,g$ run through all cusp forms of same weights of $\Gamma$.
\end{proposition}
\begin{proof}
By Corollary \ref{ccf}, we know the these $(f,g)_k$'s generate the space of continuous functions on $\mathcal{F}^*$ that vanish at cusps. 
Hence the restriction of $(f,g)_k$'s on $\mathcal{F}$ also span a dense subspace of $L^2(\mathcal{F})$.  

Note that $T((f,g)_k)=T(f\overline{g}y^k)=(T_g)^*T_f$. 
Then, by Proposition \ref{psdense}, we know these $(T_g)^*T_f$'s give a $*$-closed subalgebra which is dense in $L^2(A
_m)$. Then the claim follows from (the scalar case of) Proposition \ref{peqdense} or Theorem \ref{tgenvna}. 
\end{proof}

\section{The Von Neumann Algebras from Cusp Forms}\label{sauto}

In this section, we let $\mathbf{G}$ be a connected reductive group over $\mathbb{Q}$ such that $\mathbf{G}(\mathbb{R})=G$, the semi-simple real Lie group in the previous sections.  
Let $K$ be a maximal compact subgroup of $G$, and $\Gamma\subset G$ be a lattice. 

We consider the cusp forms for $\Gamma$ defined on $G$.  
As an analog of $SL(2,\mathbb{R})$ in the previous section, there is a natural action of the cusp forms on the holomorphic discrete series and square-integrable representations, which intertwines the actions of $\Gamma$. 
Let $H_{\pi}=L^2_{\text{holo}}(\mathcal{D},V_{\pi})$ be the holomorphic discrete series representation of $G$ associated with an irreducible representation $(\pi,V_{\pi})$ of its maximal compact subgroup $K$ in Section \ref{sholds}.  
For a given a cusp form $f$ 
for $\Gamma$, we define a Toeplitz-type operator $T_f$  associated with $f$ which acts on $H_{\pi}$. 
 
This section will be largely devoted to proving the following result. 
\begin{theorem}\label{tmain}
The commutant $A_{\pi}=L_\pi(\Gamma)'$ can be given as follows: 
\begin{center}
$\overline{\langle\{\text{span}_{f,g} T_g^{*}T_f\}\otimes \End(V_{\pi})\rangle}^{\text{w.o.}}=A_{\pi}$,
\end{center}
where $f,g$ run through holomorphic cusp forms for $\Gamma$ of same types. 

In particular, if $\dim_{\mathbb{C}}V_{\pi}=1$, the commutant can be generated by these compositions $T_g^{*}T_f$'s: 
$\overline{\langle \text{span}_{f,g} T_g^{*}T_f\rangle}^{\text{w.o.}}=A_{\pi}$.

\end{theorem}


\subsection{Automorphic forms on a real reductive group}\label{sauto1}

We review automorphic forms on real reductive Lie groups. 
A comprehensive treatment of this theory can be found in many resources, for example, see \cite{B1,BJ,HCa}, where the sides that the group acts on are different.  

Let $\mathfrak{g},\mathfrak{k}$ be the Lie algebras of $G,K$ respectively, $Z(\mathbb{R})$ be the center of $G$. 
Let $U(\mathfrak{g})$ be the universal enveloping algebra and $Z(\mathfrak{g})$ be the center of $U(\mathfrak{g})$. 
Let $\mathcal{H}=\mathcal{H}(G,K)$ be the Hecke algebra which is the convolution algebra of all $K$-finite distributions on $G$ with support in $K$ \cite{KV}.  

Let $(\pi,V)$ be a representation of $K$, and we do not assume it is finite-dimensional. 
Let $\widehat{K}$ be the set of equivalence classes of irreducible representations of $K$ and take $\rho\in \widehat{K}$. 
We denote
\begin{center}
$V(\rho)=\{v\in V|\spn_{k\in K}\pi(k)v\cong \rho\}$. 
\end{center}
By a {\it $(\mathfrak{g},K)$-module}, we mean a complex vector space $V$ with a representation $\pi$ of $\mathfrak{g}$ and $K$ which satisfies the following conditions:
\begin{enumerate}[label=(\roman*)]
    \item The space $V$ is a countable algebraic direct sum $V=\oplus_{i}V_i$ where each $V_i$ is a finite dimensional $K$-invariant vector space.
    \item For $X\in \mathfrak{p},v\in V$, we have
    \begin{center}
        $\pi(X)v=\frac{d}{dt}\pi(\exp(tX))v|_{t=0}=\lim\limits_{h\to 0}\frac{\pi(exp(hX))v-v}{h}$, 
    \end{center}
    where the limit exists.
    \item For $k\in K,X\in\mathfrak{g}$, we have $\pi(k)\pi(X)\pi(k^{-1})v=\pi(\Ad(k)X)v$. 
\end{enumerate}
Furthermore, we call it an {\it admissible $(g,K)$-modules}
if $V(\rho)$ is finite dimensional for each $\rho\in \widehat{K}$. 
Note for a smooth function $\phi\colon G\to \mathbb{C}$, there is a natural action of $K$ and $\mathfrak{g}$ (hence of $U(\mathfrak{g})$) on $\phi$. 
The complex span of $\phi$ over the action of $\mathfrak{g}$ and $K$ gives us a {\it $(\mathfrak{g},K)$-module}   \cite{Wal}. 
Moreover, we say $\phi$ is {\it slowly increasing} if there are constants $c,r\in \mathbb{R}_{>0}$ such that
\begin{center}
$|\phi(g)|\leq C\cdot \|g\|^r$,      
\end{center}
where $\|g\|=(\tr(\sigma(g)^*\sigma(g)))^{1/2}$ and $\sigma$ is a finite dimensional complex representation with finite kernel and closed image. 
This condition does not depend on the choice of the representation $\sigma$ \cite{BJ}, but the constant $C$ does. 
Following \cite{B1,BJ}, we define the automorphic forms on $G$. 
\begin{definition}\label{dafg1}
A smooth complex valued function $\phi$ on $G$ is an {\it automorphic form} for $(\Gamma,K)$ if it satisfies the following conditions:
\begin{enumerate}[label=(\roman*)]
    \item It is $\Gamma$-left invriant: $\phi(\gamma\cdot g)=\phi(g)$, $g\in G, \gamma \in \Gamma$.
    \item The right translates of $\phi$ by elements of $K$ span a finite dimensional vector space.
    \item There is an ideal $I\subset Z(\mathfrak{g})$ of finite codimension such that $x\circ \phi=0$ for all $x\in I$.  
    \item It is slow increasing. 
\end{enumerate}
\end{definition}

Note the condition (ii) is equivalent to the existence of an idempotent $\zeta\in \mathcal{H}$ such that the convolution $\phi*\zeta=\phi$. 
We let $\mathcal{A}(\Gamma,\zeta,I,K)$ be the space of all the automorphic forms of this type. 

Assume $N$ is the unipotent radical of any proper maximal parabolic $\mathbb{Q}$-subgroup of $G(\mathbb{Q})$. 
Given $\phi\in \mathcal{A}(\Gamma,\zeta,I,K)$, we call it a {\it cuspidal automorphic form}, or simply a {\it cusp form} if 
\begin{center}
$\int_{(\Gamma\cap N(\mathbb{R}))\backslash N(\mathbb{R})}\phi(n\cdot g)dn=0$, $\forall g\in G_{\mathbb{R}}$, 
\end{center}
where $dn$ is the measure on the quotient space $(\Gamma\cap N(\mathbb{R}))\backslash N(\mathbb{R})$ obtained from the measure $dg$ on $G$.  
We let  $\mathcal{A}^{0}(\Gamma,\zeta,I,K)$ be the subspace of cusp forms in  $\mathcal{A}(\Gamma,\zeta,I,K)$. 
A cusp form is bounded and square-integrable modulo $Z(\mathbb{R})\cdot \Gamma$ ($Z(\mathbb{R})$ is the center of $G$) \cite{BJ,HCa}. 

Let $\rho\colon K\to GL(V_{\rho})$ be a finite dimensional unitary representation of $K$ and $V_{\rho}$ is equipped with a Hermitian product $\langle,\rangle_{\rho}$ and hence a norm $\|\cdot\|_{\rho}$. 
We also have the following definition of Harish-Chandra \cite{HCa,BJ} of the {\it vector-valued automorphic forms} as follows. 
\begin{definition}\label{dafg2}
A smooth function $F\colon G\to V_{\rho}$ is called a $V_{\rho}$-valued automorphic form if it satisfies
\begin{enumerate}[label=(\roman*)]
    \item $F(\gamma\cdot g)=F(g)$, $g\in G, \gamma \in \Gamma$.
    \item $F(g\cdot k)=\rho(k^{-1})\cdot F(g)$, $g\in G, k\in K$.
    \item There is an ideal $I\subset Z(\mathfrak{g})$ of finite codimension such that $x\circ F=0$ for all $x\in I$.
    \item $F$ is slow increasing, i.e., $\|F(g)\|_{\rho}\leq C\cdot \|g\|^n$ for some $C>0$, for all $g\in G$.
\end{enumerate}
We denote the space of all such $V_{\rho}$-value functions by $\mathcal{A}(\Gamma,I,\rho)$ or simply $\mathcal{A}(\Gamma,\rho)$. 
Furthermore, if we also have $\int_{(\Gamma\cap N(\mathbb{R}))\backslash N(\mathbb{R})}F(n\cdot g)dn=0$,
for all $g\in G_{\mathbb{R}}$, it is called a $V_{\rho}$-valued cuspidal form or cusp form. 
This vector space is denoted by $\mathcal{A}^{0}(\Gamma,I,\rho)$ or simply $\mathcal{A}^{0}(\Gamma,\rho)$. 
\end{definition}

As in the scalar-valued case, a vector-valued cusp form is also bounded under the norm $\|\cdot\|_{\rho}$ (and hence bounded in each coordinate of $V_{\rho}$) and square-integrable modulo $Z(\mathbb{R})\cdot \Gamma$. 


Now we introduce automorphic forms on the domain $\mathcal{D}=G/K$. 
Although the definition given by A. Borel \cite{B1} involves a general cocycle $\mu$ defined on $\Gamma \times \mathcal{D}$,  
we only focus on the special case of the canonical automorphy factor $J(g,x)\colon G\times \mathcal{D}\to K_{\mathbb{C}}$ (see Section \ref{sholds}). 
We also fix a finite dimensional unitary representation $\rho\colon K\to GL(V_{\rho})$ and also denote  by $\rho$ its extension to $K_{\mathbb{C}}$. 

\begin{definition}\label{dcusptype}
A vector-valued automorphic form of type $\rho$ is a smooth function $f\colon \mathcal{D}= G/K \to V_{\rho}$ satisfying
\begin{center}
$f(\gamma x)=\rho(J(\gamma,x))\cdot f(  x)$, $\forall x\in \mathcal{D},\forall \gamma\in \Gamma$.  
\end{center}
We denote the space of such functions by $\mathcal{A}_{\mathcal{D}}(\Gamma,J)$.
\end{definition}

\begin{example}
Let $G=SL_2(\mathbb{R})$, $K=SO(2)$ and $\Gamma\subset G$ is the modular group. 
Let $J(g,z)=cz+d$ for $g=\bigl(\begin{smallmatrix} a &b \\ c & d \\ \end{smallmatrix}\bigr)\in G$. 
We also take the irreducible representation $\rho=\rho_m\colon K\cong S^1 \to S^1$ given by $z\mapsto z^m$. 
Then a function $f\colon \mathbb{H}\to \mathbb{C}$ satisfying
\begin{center}
$f(\frac{az+b}{cz+d})=(cz+d)^m f(z)$ 
\end{center}
for all $\bigl(\begin{smallmatrix} a &b \\ c & d \\ \end{smallmatrix}\bigr)\in \Gamma$ gives us a classical modular form of weight $m$ (or equivalently of type $\rho_m$)for $\Gamma$ (with some holomorphy conditions required, see \cite{Sh}). 
\end{example}

\begin{remark}
$\mathcal{A}_{\mathcal{D}}(\Gamma,J)$ can be related with automorphic forms $\mathcal{A}(\Gamma,\rho)$ on $G$ in Definition \ref{dafg1} and \ref{dafg2}.
Given $F\in \mathcal{A}(\Gamma,\rho)$,
the map given by $\Phi(F)(\dot{g})=f(\dot{g})=\rho(J(g,0))F(g)$ is well-defined and $\Phi(\mathcal{A}(\Gamma,\rho))\subset \mathcal{A}_{\mathcal{D}}(\Gamma,J)$. 
Take $g_1,g_2\in G$ such that $\dot{g_1}=\dot{g_2}$ and we assume $g_1=g_{2}k$ for some $k\in K$. 
Then we have
\begin{equation*}
\begin{aligned}
\rho(J(g_1,0))F(g_1)&=\rho(J(g_2k,0))F(g_2k)=\rho(J(g_2k,0))\rho(k)^{-1}F(g_2)\\
&=\rho(J(g_2,0))\rho(J(k,0))\rho(k)^{-1}F(g_2)\\
&=\rho(J(g_2,0))\rho(k)\rho(k)^{-1}F(g_2)=\rho(J(g_2,0))F(g_2),
\end{aligned}
\end{equation*}
where we apply the cocycle condition of $J$ and the fact $J(k,0)=k$ for $k\in K$. 
We also have
\begin{center}
$f(\dot{\gamma g})=\rho(J(\gamma g,0))F(\gamma g)=\rho(J(\gamma ,\dot{g})J(g,0))F(g)
=\rho(J(\gamma,\dot{g}))f(\dot{g})$.
\end{center}
So $f\in \mathcal{A}_{\mathcal{D}}(\Gamma,J)$. 
\end{remark}

For a general cocyle $\mu$, we may not be able to define an inverse map from  $\mathcal{A}_{\mathcal{D}}(\Gamma,\mu)$ to $\mathcal{A}(\Gamma,\rho)$ since the cocyle $\mu$ on $\Gamma$ cannot always be extended to $G$. 
An example on $\Gamma_0(4)$ with half-integral weight was given by G. Shimura \cite{Sh1}. 
Recall there is a smooth embedding $i\colon \mathcal{D}= G/K\hookrightarrow NA\subset G$ and we denote this map from $\mathcal{D}$ to $G$ by $i(z)=g_z$ (see Section \ref{sholds}). 
Please note we have $z=\dot{g_z}=g_z\cdot K$ as a coset in $G/K$. 

Now we consider a special kind of automorphic forms called {\it Poincar\'{e} series}. 
The Poincar\'{e} series were first constructed for automorphic forms on a Lie group by Poincar\'{e}. 
Intuitively, for $SL(2,\mathbb{R})$, Poincar\'{e} series apply group averages of an infinite sum, which is a natural way to construct functions invariant under the automorphy action of the modular group. 

We have the following result proved by Harish-Chandra (see \cite{BB} Theorem 5.4) and R. Godement \cite{Go62}. 
Please the side that the group acts on.  
\begin{theorem}\label{tpoin1}
Let $\Gamma$ be a discrete subgroup of $G$ and $V$ be a finite-dimensional complex vector space. 
Let $\tilde{f}\colon G\to V$ be a function in $L^1(G)\otimes V$, which is $Z(\mathfrak{g})$-finite and is $K$-finite on the left. 
Then the series
\begin{center}
$P_{\tilde{f}}(g)=\sum_{\gamma\in \Gamma}\tilde{f}(\gamma\cdot g )$, $P_{\|\tilde{f}\|}(g)=\sum_{\gamma\in \Gamma}\|\tilde{f}(\gamma\cdot g)\|$, 
\end{center}
are absolutely and uniformly convergent on compact subsets and are bounded on $G$. 

Furthermore, if $\tilde{f}$ is of finite type on the right instead, $P_{\tilde{f}}$ is absolutely and uniformly convergent on compact sets, but not necessarily bounded. 
\end{theorem}

The series $P_{\tilde{f}}$ are called a {\it Poincar\'{e} series on the group} $G$. 
We are more interested in the following Poincar\'{e} series defined on the bounded symmetric domain $\mathcal{D}$. 
For the canonical autormorphy factor $J\colon G\times \mathcal{D}\to K_{\mathbb{C}}$ and a representation $(\rho,V_{\rho})$ of $K$, we associate a function $f\colon \mathcal{D}\to V_{\rho}$ to $\tilde{f}\colon G\to V\rho$ by
\begin{center}
$f(\dot{g})=\rho(J(g,0))\tilde{f}(g)$. 
\end{center}
We can show $f$ is holomorphic on $\mathcal{D}$ if and only if $Y\circ \tilde{f}=0$ for all $Y\in \mathfrak{p}^-$. 

Now we consider the action of $G$ on $G/K=\mathcal{D}\subset \mathfrak{p}^+$. 
For any $z\in \mathcal{D}$, we let $J_{\mathcal{D}}(g,z)\in \mathbb{C}$ be the determinant of the Jacobian of $z\mapsto g(z)$. 
\begin{lemma}[\cite{Sa} II.5.3.]\label{ljj}
The Jacobian of $z\mapsto g(z)$ is the adjoint representation $\ad$ of the canonical automorphy factor $J$, i.e.,
\begin{center}
$\Jac(z\mapsto g(z))=\ad_{\mathfrak{p}^+}(J(g,z))$. 
\end{center}
where $\ad_{\mathfrak{p}^+}$ is the restriction of $\ad$ on $\mathfrak{p}^+$. 
\end{lemma}

Hence we have $J_{\mathcal{D}}(g,z)=\det(\ad_{\mathfrak{p}^+}(J(g,z)))$. 
We now consider {\it Poincar\'{e} series} on the domain $\mathcal{D}$ and have the result as following  \cite{B1,Cartan58}. 
\begin{theorem}\label{tpcseries}
Let $f$ be a polynomial function on $\mathcal{D}$, $m\geq 4$ be an integer. 
Then the series 
\begin{center}
$P_{m,f}(z)=\sum_{\gamma\in \Gamma}J_{\mathcal{D}}(\gamma,z)^m f(\gamma\cdot z)$
\end{center}
converges absolutely and uniformly on compact sets. 
It defines a holomorphic automorphic form of weight $m$, i.e.,
\begin{center}
$P_{m,f}(z)=J_{\mathcal{D}}(\gamma,z)^mP_{m,f}(\gamma z)$. 
\end{center}
The function $\widetilde{P}_{m,f}\colon g\mapsto J_{\mathcal{D}}(g,0)^{-m}P_{m,f}(g\cdot 0)$ is bounded on $G$. 
\end{theorem}
  
Indeed, these series are sufficient to separate the points. 
\begin{proposition}\label{ppoinsep}
Let $z_1,\dots,z_N$ be $\Gamma$-inequivalent points in $\mathcal{D}$. 
Take a set of points $c_1,\dots,c_N\in \mathbb{C}$. 
For $m$ sufficiently large, we can find a polynomial $f$ such that $P_{f,m}(z_i)=c_i$ for all $1\leq i\leq N$.
\end{proposition}
\begin{proof}
Let us consider the linear map given by
\begin{center}
    $E_{\vec{z}}\colon f\mapsto (P_{m,f}(z_1),\dots,P_{m,f}(z_N))$.
\end{center}
Its image is a linear subspace of $\mathbb{C}^N$. 
We will show this map is surjective for sufficiently large $m$. 

Take some $0<u<1$. 
By \cite{BB} Lemma 5.8, we know the function $g\mapsto |J_{\mathcal{D}}(g,0)|^a$ is in $L^1(G)$ if $a\geq 2$. 
Then, by \cite{BB} Theorem 5.10, we conclude $\sum_{\gamma\in \Gamma}J_{\mathcal{D}}(\gamma,z)^2$ converges absolutely and uniformly on compact sets. 
Hence the set $\Gamma_{u,i}=\{\gamma\in \Gamma||J_{\mathcal{D}}(\gamma,z_i)|>u\}$ is a finite. 
Now let $f$ be a polynomial on $\mathfrak{p}^+$ such that
\begin{center}
$ f(z_i)=c_i$ and $f(\gamma z_i)=0$ for $\gamma \in \Gamma_{u,i}$, $1\leq i\leq N$. 
\end{center}
Note $J_{\mathcal{D}}(e,z_i)^l f(z_i)=f(z_i)=c_i$, we have $|P_{l,f}(z_i)-c_i|\leq \sum_{\gamma \notin \Gamma_{u,i}}|u^lf(\gamma z_i)|$ which converges to $0$ as $l\to \infty$. 
So for any $\varepsilon>0$, there exists an integer $m(\varepsilon)$ such that $\|E(f)-\vec{c}\|_{\mathbb{C}^N}<\varepsilon$ if $m>m(\varepsilon)$. 
Now let $\vec{c}$ runs through the standard basis $\{e_1,\dots,e_N\}$ of  $\mathbb{C}^N$ and let $\varepsilon$ be small enough, the argument above implies the map $E$ contains a basis if $m$ is sufficiently large. 
Hence $E$ is surjective. 
\end{proof}

\subsection{Actions of cuspidal automorphic forms}\label{ssintwop}

Take a cusp form $F\in \mathcal{A}^0(\Gamma,\rho)$ and let $f=\Phi(F)$, which is a vector-valued automorphic form on the domain $\mathcal{D}$. 
We call the function $f$ of such type a {\it cusp form} on $\mathcal{D}$ and denote them by $\mathcal{A}_\mathcal{D}^0(\Gamma,\rho)$. 

\begin{lemma}
\label{lMf}
Let $\phi\in L^2(\mathcal{D},V_{\pi})$, then the map
\begin{center}
$M_f\colon f(z)\mapsto M_{f}(\phi)(z)=f(z)\otimes \phi(z)$
\end{center}
is a well-defined bounded map with image in $L^2(\mathcal{D},V_{\rho\otimes\pi})$
\end{lemma}
\begin{proof}
Let us consider the norm $\| M_f(\phi)\|_{L^2(\mathcal{D},V_{\rho\otimes\pi})}$. 
We have
\begin{equation*}
\begin{aligned}
&\| M_f(\phi)\|^2=\langle f\otimes \phi,f\otimes \phi\rangle_{L^2}\\
=&\int_{\mathcal{D}}\langle (\rho\otimes\pi)(\kappa(w,w)^{-1})f(w)\otimes \phi(w),f(w)\otimes \phi(w) \rangle_{V_{\rho\otimes\pi}}d\mu(w)\\
=&\int_{\mathcal{D}}\langle \rho(\kappa(w,w)^{-1})f(w),f(w)\rangle_{V_{\rho}} \cdot  \langle \pi(\kappa(w,w)^{-1})\phi(w),\phi(w)\rangle_{\pi}d\mu(w). 
\end{aligned}
\end{equation*}
Suppose $w\in \mathcal{D}\cong G/K$ has a representative $\dot{g}$ with $g\in G$. 
Note $f(\dot{g})=\rho(J(g,0))F(g)$ and $\kappa(\dot{g},\dot{g})=J(g,0)\overline{J(g,0)^{-1}}$. 
Following Remark \ref{r1}, one has $\rho(\kappa(\dot{g},\dot{g}))=\rho(J(g,0))\rho(J(g,0))^{*}$. 
(Note that $F(g)$ is independent of the choice of the representative $\dot{g}$ in the coset $w\in \mathcal{D}\cong G/K$.) 
We have
\begin{equation*}
\begin{aligned}
&\langle\rho(\kappa(w,w)^{-1})f(w),f(w)\rangle_{V_{\rho}}\\
=&\langle(\rho(J(g,0))^{-1})^* \rho(J(g,0)^{-1})\rho(J(g,0))F(g),\rho(J(g,0))F(g)\rangle_{V_{\rho}}\\
=&\langle F(g),F(g)\rangle_{V_{\rho}}. 
\end{aligned}
\end{equation*}
Note $F$ is bounded on $G$ since it is a cusp form, i.e., there is a positive constant $C_F$ such that $\langle F(g),F(g)\rangle_{V_{\rho}}\leq C_F$ for all $g\in G$. 
So we get $\| M_f(\phi)\|^2_{L^2(\mathcal{D},V_{\rho\otimes\pi})}\leq C_F\cdot \| \phi\|^2_{L^2(\mathcal{D},V_{\pi})}$. 
\end{proof}
\begin{remark}
The tensor product above is pointwise defined. 
Indeed, the vector-valued function $\rho(J(g,0)^{-1})f(\dot{g})$ is essentially bounded so that $f(z)\otimes \phi(z)$ is still in the Hilbert space $L^2(\mathcal{D},V_{\rho\otimes\pi})$. 

In general, the tensor Hilbert space $H_{\rho}\otimes H_{\pi}$ is an infinite direct sum of discrete series representations, which is much larger than $H_{\rho\otimes \pi}$.   
J. Repka gave a clear description of the decomposition of arbitrary tensor products of these holomorphic discrete series representations \cite{Re79}. 
\end{remark}

Now we define a Toeplitz-type operator on the holomorphic discrete series. 
Recall that $P_{\pi}$ is the orthogonal projection from $L^2(\mathcal{D},V_{\pi})$ to $H_{\pi}=L_{\text{hol}}^2(\mathcal{D},V_{\pi})$.
Let 
\begin{center}
$T_f\colon H_{\pi}\to H_{\rho\otimes \pi}$    
\end{center} 
be the map defined by
\begin{center}
$T_f(\phi)=P_{\rho\otimes \pi}M_f P_{\pi}(\phi)=P_{\rho\otimes \pi}M_f (\phi)=P_{\rho\otimes \pi}(f(z)\otimes \phi(z))$. 
\end{center}
where $\phi\in H_{\pi}$.

\begin{proposition}\label{pcuspcom}
We have $T_f\in B(H_{\pi},H_{\rho\otimes \pi})$ which commutes with the action of $\Gamma$, i.e., $T_{f}L_{\pi}(\gamma)=L_{\rho\otimes\pi }(\gamma)T_f$, for all $\gamma\in \Gamma$.  
\end{proposition}
\begin{proof}
The boundedness follows from Lemma \ref{lMf}. 
Take $\phi,\eta \in H_{\pi}$, $\psi\in H_{\rho}$. 
We have 
\begin{equation*}
\begin{aligned}
&\langle L_{\rho\otimes\pi }(\gamma)T_f(\phi(z)),(\psi\otimes\eta)(z)\rangle_{H_{\rho\otimes\pi}}\\
=&\langle L_{\rho\otimes\pi }(\gamma)P_{\rho\otimes\pi}M_f(\phi(z)),(\psi\otimes\eta)(z)\rangle_{H_{\rho\otimes\pi}}\\
=&\langle P_{\rho\otimes\pi}L_{\rho\otimes\pi }(\gamma)(f\otimes\phi)(z),(\psi\otimes\eta)(z)\rangle_{L^2}\\
=&\langle (L_{\rho }(\gamma)(f)\otimes L_{\pi }(\gamma)(\phi))(z),(\psi\otimes\eta)(z)\rangle_{L^2}\\
=&\langle T_f L_{\pi }(\gamma)(\phi)(z),(\psi\otimes\eta)(z)\rangle_{H_{\rho\otimes\pi}}
\end{aligned}
\end{equation*}
where we apply Proposition \ref{pinvariant} and use the fact $L_{\rho}(\gamma)f=f$.  
\end{proof}

Take another cusp form $H\in \mathcal{A}^0(\Gamma,\rho)$ and let $h(z)=\Phi(H)$. 
We define the following function $(f,h)_{\rho}$ on $\mathcal{D}$: 
\begin{center}
$(f,h)_{\rho}(z)=\langle \rho(\kappa(z,z)^{-1})f(z),h(z)\rangle_{V_{\rho}}$.  
\end{center}
Indeed, this is a generalization of integrand $f(z)\overline{h(z)}y^k$ in the Petersonn inner product 
\begin{center}
$(f,h)=\int_{\Gamma\backslash\mathbb{H}}f(z)\overline{h(z)}y^{k-2}dxdy$
\end{center}
of cusp forms $f,h\in S_k(\Gamma)$ of weight $k$ on $SL(2,\mathbb{R})$. 
For cusp forms on real Lie groups, we can show $(f,h)_{\rho}(z)=\langle F(g_z),H(g_z)\rangle_{V_{\rho}}$ where $g_z$ is a representative of $z\in \mathcal{D}\cong G/K$. 

\begin{corollary}
\label{cTeo}
The composite operator $T_{h}^{*}T_f \in B(H_{\pi})$ commutes with the action of $\Gamma$, i.e., $T_{h}^{*}T_fL_{\pi}(\gamma)=L_{\pi}(\gamma)T_{h}^{*}T_f$ for $\gamma\in \Gamma$. 
Moreover, if either one of $f$ or $h$ is holomorphic, we have
\begin{center}
$T_{h}^{*}T_f\phi(z)=T_{(f,h)_{\rho}}\phi(z)=T_{\langle F(g_z),H(g_z)\rangle_{V_{\rho}}}\phi(z)$
\end{center}
for any $\phi(z)\in H_{\pi}$. 
\end{corollary}
\begin{proof}
The first claim is straightforward by Proposition \ref{pcuspcom}. 
By Remark \ref{r1}, we have
$\rho(\kappa(z,z))=\rho(J(g_z,0))\rho(J(g_z,0)^*)=\rho(J(g_z,0))\rho(J(g_z,0))^*$. 
As $F(g_z)=\rho(J(g_z,0)^{-1})f(z)$, we obtain: 
\begin{equation*}
\begin{aligned}
(f,h)_{\rho}(z)&=\langle \rho(\kappa(z,z)^{-1})f(z),h(z)\rangle_{V_{\rho}}\\
&=\langle \rho(J(g_z,0)^{-1})f(z),\rho(J(g_z,0)^{-1})h(z)\rangle_{V_{\rho}}\\
&=\langle F(g_z),H(g_z)\rangle_{V_{\rho}},
\end{aligned}
\end{equation*}
which is bounded since $F,H$ are bounded in each coordinate. 
Hence $\langle F(g_z),H(g_z)\rangle_{V_{\rho}}\in L^{\infty}(\mathcal{D})$ that makes the associated Toeplitz operator $T_{(f,h)_{\rho}}$ well-defined. 

Now we assume $f$ is holomorphic and take an arbitrary $\eta\in H_{\pi}$. 
Note $f(z)\otimes \phi(z)\in H_{\rho\otimes \phi}$. 
We have: 
\begin{equation*}
\begin{aligned}
&\langle T_h^* T_f\phi,\eta \rangle_{H_{\pi}}=\langle  T_h^*f\otimes\phi,\eta \rangle_{H_{\pi}}=\langle f\otimes \phi,T_h\eta\rangle_{H_{\rho\otimes\pi}}=\langle f\otimes \phi,h\otimes \eta\rangle_{L^2}\\=&\int_{\mathcal{D}}\langle (\rho\otimes\pi)(\kappa(w,w)^{-1})(f\otimes \phi)(w),(h\otimes \eta)(w) \rangle_{V_{\rho\otimes\pi}}d\mu(w)\\
=&\langle M_{(f,h)_{\rho}}\phi,\eta\rangle_{L^2}=\langle T_{(f,h)_{\rho}}\phi,\eta\rangle_{H_{\pi}} 
\end{aligned}
\end{equation*}

\end{proof}


\subsection{Baily-Borel compactification}

We review some basic facts of the Baily-Borel compactification of the quotient space $\Gamma\backslash\mathcal{D}$. 
It is a generalization of the compactification of the fundamental domain $SL(2,\mathbb{Z})\backslash \mathbb{H}$, or equivalently, 
\begin{center}
$SL(2,\mathbb{Z})\backslash SL(2,\mathbb{R})/SO(2)\cong\{z\in\mathbb{H}|~|z|>1,|\text{Re}(z)|<\frac{1}{2}\}$     
\end{center}
to a general Lie group $G$ and an arithmetic subgroup $\Gamma$. 
More details can be found in \cite{AMRT,BB,Sa} with the group actions on different sides. 


We regard $\mathcal{D} \subset \mathfrak{p}^+=\mathbb{C}^N$ as a smooth manifold with the natural action of $G$ by the Harish-Chandra realization. 
The action also extends to the closure $\overline{\mathcal{D}}\subset \mathfrak{p}^+$. 
We say a real affine hyperplane $H\subset \mathbb{C}^N$ is a {\it supporting hyperplane} if $H\cap \overline{\mathcal{D}}\neq \emptyset$ and $H\cap \mathcal{D}=\emptyset$. 
Let $H$ be such a supporting plane and denote its intersection with the closure $\overline{\mathcal{D}}$ by $\overline{F}$, i.e., $\overline{F}=H\cap \overline{\mathcal{D}}$.  
Furthermore, there is a minimal complex affine subspace $L\subset \mathbb{C}^m$ contains $\overline{F}$. 
The {\it boundary component} $F$ is a nonempty open subset of $L$ whose closure is $\overline{F}$, which is also a bounded symmetric domain in $L$.  
The detailed construction of boundary components is based on the $\mathbb{R}$-roots and $\mathbb{Q}$-roots of $\mathbf{G}$  \cite{AMRT,BB}. 

One should keep in mind that $\mathcal{D}$ itself is an improper rational boundary component. 
For each rational boundary component $F$, there is a canonical projection $\sigma_F\colon \mathcal{D}\to F$. 

For a boundary component $F$,  its normalizer is defined by
\begin{center}
$N_G(F)=\{g\in G|g\cdot F\subset F\}$,
\end{center}
where the action is induced from that on $\overline{\mathcal{D}}$. 
It is well-known $N_G(F)$ is a parabolic subgroup of $G$. 
We call $F$ a {\it rational boundary component} if $N_G(F)$ is defined over $\mathbb{Q}$ as a subgroup of the linear algebraic group $\mathbf{G}$ (see \cite{BB}.3 and \cite{AMRT}.III).  
There are countably many rational boundary components. 


Now we let $\mathcal{D}^*$ be the union of $\mathcal{D}$ and its rational boundary components, equipped with the {\it Satake topology}, which is the unique topology with some properties related to the arithmetic group $\Gamma$ (see \cite{AMRT} III. or \cite{BB} Theorem 4.9). 
Let $\mathcal{F}=\Gamma\backslash \mathcal{D}$ be the fundamental domain. 
\begin{definition}
The Baily-Borel compactification $\mathcal{F}^*$ of $\mathcal{F}$ is defined to be the quotient 
\begin{center}
$\mathcal{F}^*=\Gamma\backslash \mathcal{D}^*$,
\end{center}
equipped with the quotient topology. 
\end{definition}

More precisely, we have the following theorem for Baily-Borel compactification, see \cite{AMRT}.III.6 and \cite{BB} Theorem 10.11. 
\begin{theorem}[Baily-Borel Compactification]\label{tBB1}
There exists an isomorphism 
$\theta\colon \mathcal{F}^*\to \mathbf{P}_{\mathbb{C}}^N$ of $\mathcal{F}^*$ onto a normally projective subvariety of $\mathbf{P}_{\mathbb{C}}^N$ such that 
\begin{enumerate}
    \item $\mathcal{F}^*$ a compact Hausdorff space containing $\mathcal{F}$ as an open dense subset,
    \item $\mathcal{F}^*$ is a finite union of subspaces of the form
\begin{center}
    $\Gamma_F\backslash F$,
\end{center}
where $F$ is a rational boundary component and $\Gamma_F=\Gamma\cap N_G(F)$,
\item The closure of $\Gamma_F\backslash F$ is the union of $\Gamma_F\backslash F$ and subspaces $\Gamma_{F'}\backslash F'$ of strictly smaller dimension. 
\end{enumerate}
\end{theorem}

From now on, we will denote the compactification as
\begin{center}
$\mathcal{F}^*=V_0\cup V_1\cup\dots \cup V_t$,
\end{center}
where $V_0=\mathcal{F}$ and $V_i=\Gamma_{F_i}\backslash {F_i}$ for some rational boundary component $F_i$ with $0\leq i\leq t$. 
It can also be proved that $\dim(\mathcal{F}^*-\mathcal{F})\leq \dim(\mathcal{F})-2$ if $G$ has no normal $\mathbb{Q}$-subgroup of dimension $3$. 

Assume $h$ is an integral automorphic form  on $\mathcal{D}$ of weight $l$ (in the sense of \cite{BB} Section 8.5 and notice side for action), i.e.,  $h(z)=J_{\mathcal{D}}(\gamma,z)^l h(\gamma\cdot z)$ for all $\gamma\in \Gamma$.  
For each rational boundary component $F$, it has an extension to an automorphic form for $\Gamma_F$ on $F$, which we denote by $\Phi_F(h)$.

We have the following result on the extension of Poincar\'{e} series on the domain $\mathcal{D}$ in Section \ref{sauto1} to rational boundary components. 

\begin{proposition}\label{pPSvan}
Let $P_{m,f}(z)=\sum_{\gamma\in \Gamma}J_{\mathcal{D}}(\gamma,z)^m f(\gamma\cdot z)$ be the Poincar\'{e} series associated with a polynomial $f$ on $\mathcal{D}$. 
We have 
\begin{center}
$\Phi_F(P_{m,f})=0$    
\end{center}
for all proper rational boundary component $F$. 
\end{proposition}
\begin{proof}
Note $\mathcal{D}$ is a rational boundary component of the highest dimension. 
For any proper rational boundary component $F$, we have $\dim F\leq \dim \mathcal{D}$ and $F\subsetneq \Gamma\cdot \mathcal{D}$. 

Following \cite{BB} Theorem 8.6, we have $\Phi_F(P_{m,f})=0$. 
\end{proof}

\begin{remark}\label{rteopc}
This also implies $P_{m,\phi}$ is cuspidal (see \cite{BB} Section 8.10). 
\end{remark}

\subsection{The result on real Lie groups}

Observe the function $g\mapsto J_{\mathcal{D}}(g,0)^{-m}P_{m,f}(g\cdot 0)$ is bounded (Theorem \ref{tpcseries}),  
as in the proof of Lemma \ref{lMf}.
We are able to associate to each $P_{m,f}$ a well-defined Toeplitz operator $T_{P_{m,f}}$. 
We will focus on the Toeplitz operators of this type for the proof of Theorem \ref{tmain}.

Recall two cusp forms $f,g$ on the bounded symmetric domain $\mathcal{D}$ are called of a same type $\rho$ if both $f,g$ take values in a representation $V_{\rho}$ of $K$ and $f(\gamma x)=\rho(J(\gamma,x))\cdot f(x), g(\gamma x)=\rho(J(\gamma,x))\cdot g(x)$ in Definition \ref{dcusptype}.

\begin{proof}[Proof of Theorem~\ref{tmain}]
By Corollary \ref{cTeo}, if $f,g$ are cusp forms of type $(\rho,V_{\rho})$ and at least one of them is holomorphic, we know the composite operator $T_g^*T_f$ is just the Toeplitz operator $T_{(f,g)_{\rho}}$ associated with the essentially bounded function $(f,g)_{\rho}(w)=\langle \rho(\kappa(w,w)^{-1})f(w),g(w))\rangle_{V_{\rho}}$. 
By Corollary \ref{cTfgen}, it suffices to show these $(f,g)_{\rho}$'s span a dense subspace of $L^{\infty}(\mathcal{F})$ (or equivalently $C(\mathcal{F})$, $L^{2}(\mathcal{F})$).  

By Theorem \ref{tBB1}, we know $\mathcal{F}^*$ is a compact Hausdorff space which contains $\mathcal{F}$ as a dense open subset. 
Let us consider the quotient space $\mathcal{F}_1$ of $\mathcal{F}^*$ by identifying all the elements do not belong to $\mathcal{F}$ (which form a closed subset of $\mathcal{F}^*$ ), i.e.,
\begin{center}
$\mathcal{F}_1=\mathcal{F}^*/(\mathcal{F}^*\backslash\mathcal{F})$.
\end{center}
This is also the disjoint union of $\mathcal{F}$ and a single point, denoted as $\{\text{pt}\}$ (which represents all the proper boundary components), i.e., $\mathcal{F}_1=\mathcal{F}\sqcup\{\text{pt}\}$. 

Let us consider the Poincar\'{e} series $P_{m,\phi}$ for a polynomial $\phi$ on $\mathcal{D}$. 
By Proposition \ref{pPSvan}, $\Phi_F(P_{m,\phi})=0$ for any proper boundary components $F$. 
Hence every $P_{m,\phi}$ gives a well-defined function on $\mathcal{F}_1$ which vanishes at $\text{pt}$. 
Then it suffices to consider the functions of the type $(P_{m,\phi},P_{m,\psi})_{\rho}$.  

Take any two distinct points $z_1,z_2\in \mathcal{F}_1$ and consider the following two cases: (i) $z_1=\text{pt},z_2\in \mathcal{F}$, or (ii) $z_1,z_2\in \mathcal{F}$. 
In either of the two cases, by Proposition \ref{ppoinsep}, there is a polynomial $\phi$ such that $P_{m,\phi}(z_1)\neq P_{m,\phi}(z_2)$ for some $m$. 
So $(P_{m,\phi},P_{m,\psi})(z_1)_{\rho}\neq (P_{m,\phi},P_{m,\psi})_{\rho}(z_2)$ for a suitable $\psi$ such that $P_{m,\psi}(z_2)\neq 0$. 
Note $\mathcal{F}_1$ is compact and Hausdorff, by Stone-Weierstrass Theorem, the forms $(P_{m,\phi},P_{m,\psi})_{\rho}$ generate the space
\begin{center}
$\{h\colon \mathcal{F}_1\to \mathbb{C}|h\text{~is~ continuous},~h(\text{pt})=0\}$.    
\end{center} 
Hence their restriction to $\mathcal{F}$ is dense in $L^{\infty}(\mathcal{F})$, which completes the proof. 
\end{proof}

\begin{remark}
This result may also be done by showing the rational functions on the projective variety $\mathcal{F}^*$ separate points, 
which can be reduced to the proof that $\mathcal{F}^*$ is an integral separated scheme. 

Assume $G$ is a connected semi-simple linear algebraic group over $\mathbb{R}$ and a lattice $\Gamma$ is Zariski-dense in $G$.  
We can show $\Gamma$ is an ICC group (see \cite{GHJ} 3.3.b). 
This gives a large family of the cases that $A_{\pi}$ is a $\text{II}_1$ factor. 
\end{remark}

\bibliographystyle{abbrv}
\typeout{}
\bibliography{ms} 
\end{document}